\documentclass[a4paper]{amsart}

\usepackage[utf8]{inputenc}
\usepackage{amsthm, amssymb, amsmath, amsfonts}
\usepackage{url}

\newtheorem{thm}{Theorem}[section]
\newtheorem{prop}[thm]{Proposition}
\newtheorem{lem}[thm]{Lemma}

\theoremstyle{remark}
\newtheorem{rem}[thm]{Remark}

\renewcommand{\le}{\leqslant}
\renewcommand{\ge}{\geqslant}

\newcommand{\mcl}{\mathcal}
\newcommand{\E}{\mathbb{E}}
\newcommand{\EE}{\mathbf{E}}
\newcommand{\EEo}{\mathbf{E}^\omega}

\newcommand{\N}{\mathbb{N}}

\renewcommand{\l}{\mathsf{d}}
\newcommand{\Ll}{\left}
\newcommand{\Rr}{\right}
\newcommand{\1}{\mathbf{1}}
\newcommand{\R}{\mathbb{R}}

\newcommand{\Z}{\mathbb{Z}}
\renewcommand{\P}{\mathbb{P}}
\newcommand{\PP}{\mathbf{P}}
\newcommand{\PPo}{\mathbf{P}^\omega}
\newcommand{\ov}{\overline}
\newcommand{\td}{\tilde}
\newcommand{\eps}{\varepsilon}
\def\d{{\mathrm{d}}}

\newcommand{\mfk}{\mathfrak}

\newcommand{\la}{\langle}
\newcommand{\ra}{\rangle}
\newcommand{\dr}{\partial}
\renewcommand{\o}{\omega}

\DeclareMathOperator{\arsinh}{arsinh}

\title[martingale CLT and quantitative homogenization]{Kantorovich distance in the martingale CLT and quantitative homogenization of parabolic equations with random coefficients}
\author{Jean-Christophe Mourrat}

\address{Ecole polytechnique fédérale de Lausanne, institut de mathématiques, station 8, 1015 Lausanne, Switzerland}

\begin{document}
\begin{abstract}

The article begins with a quantitative version of the martingale central limit theorem, in terms of the Kantorovich distance. This result is then used in the study of the homogenization of discrete parabolic equations with random i.i.d.\ coefficients. For smooth initial condition, the rescaled solution of such an equation, once averaged over the randomness, is shown to converge polynomially fast to the solution of the homogenized equation, with an explicit exponent depending only on the dimension. Polynomial rate of homogenization for the averaged heat kernel, with an explicit exponent, is then derived. Similar results for elliptic equations are also presented.

\bigskip


\medskip


\end{abstract}
\maketitle
%
%
%
%
%
%
%
%
\section{Introduction}
\label{s:intro}
\setcounter{equation}{0}
\subsection{Main results}
The main goal of this article is to give quantitative estimates in the homogenization of discrete divergence form operators with random coefficients. Writing $\mathbb{B}$ for the set of edges of $\Z^d$, we let $\omega = (\omega_e)_{e \in \mathbb{B}}$ be a family of i.i.d.\ random variables, assumed to be uniformly bounded away from $0$ and infinity, and whose joint distribution will be written $\P$ (with associated expectation $\E$). The operator whose homogenization properties we wish to investigate is
\begin{equation}
\label{defLom}
L^\omega f(x) = \sum_{y \sim x} \omega_{x,y} (f(y)-f(x)) \qquad (x \in \Z^d),
\end{equation}
where we write $y \sim x$ when $x,y \in \Z^d$ are nearest neighbours. For a bounded continuous $f : \R^d \to \R$, we consider $u^{(\eps)}$ the solution of
\begin{equation*}
\label{parabeqeps}
\tag{DPE$^\omega_\eps$}
\left\{
\begin{array}{ll}
\displaystyle{\frac{\dr u^{(\eps)}}{\dr t}  = L^\omega u^{(\eps)}    } & \text{on } \R_+ \times \Z^d, \\
\\
\displaystyle{   u^{(\eps)}(0,\cdot) = f(\eps\  \cdot)   } & \text{on } \Z^d,
\end{array}
\right.	
\end{equation*}
and $u_\eps(t,x)  = u^{(\eps)}(\eps^{-2} t,\lfloor \eps^{-1} x \rfloor)$. There exists a symmetric positive-definite matrix $\ov{A}$ (independent of $f$) such that the function $u_\eps$ converges, as $\eps$ tends to $0$, to the function $\ov{u}$ solution of
\begin{equation*}
\label{cparabeq}
\tag{CPE}
\left\{
\begin{array}{ll}
\displaystyle{\frac{\dr \ov{u}}{\dr t}  = \frac{1}{2} \nabla \cdot \Ll( \ov{A}\nabla \ov{u} \Rr)   } & \text{on } (0,+\infty) \times \R^d, \\
\\
\displaystyle{   \ov{u}(0,\cdot) = f   } & \text{on } \R^d.
\end{array}
\right.	
\end{equation*}
The notions of being a solution to \eqref{parabeqeps} or \eqref{cparabeq}, and of the convergence of $u_\eps$ to $\ov{u}$, will be made precise later on. For a bounded measurable function $f : \R^d \to \R$ and $\alpha = (\alpha_1, \ldots, \alpha_d) \in \N^d$, we call 
$$
\dr_{x_1^{\alpha_1} \ldots x_d^{\alpha_d}} f = \frac{\dr^{|\alpha|_1} f}{\dr{x_1^{\alpha_1}} \cdots\dr{x_d^{\alpha_d}}} \qquad \textstyle{(|\alpha|_1 = \sum_j \alpha_j)}
$$
a \emph{weak derivative} of order $|\alpha|_1$, where the derivative is understood in the sense of distributions. 

Here and below, we write $\lfloor x \rfloor$ for the integer part of $x$, $a \wedge b = \min(a,b)$, $a \vee b = \max(a,b)$, $\log_+(x) = \log(x) \vee 1$, and $|\xi|$ for the $L^2$ norm of $\xi \in \R^d$. The main purpose of this paper is to prove the following theorems.
\begin{thm}
\label{t:main}
Let $m = \lfloor d/2 \rfloor + 3$ and $\delta > 0$. There exist constants $C_\delta$ (which may depend on the dimension) and $q$ such that, if the weak derivatives of order $m$ of $f$ are in $L^2(\R^d)$, then for any $\eps > 0$, $t > 0$ and $x \in \R^d$, one has
\begin{multline}
\label{e:main}
\Ll|\E[u_\eps(t,x)] - \ov{u}(t,x) \Rr| \\
\le \sum_{j=1}^d \|\dr_{x_j} f \|_\infty \ \eps +  C_\delta \  (t+\sqrt{t})  \Ll(\|f\|_2 +   \sum_{j=1}^d \|\dr_{x_j^m} f\|_2 \Rr) \ \Psi_{q,\delta}\Ll(\frac{\eps^2}{{t}}\Rr),
\end{multline}
where
\begin{equation}
\label{defpsi}
\Psi_{q,\delta}(u) = 
\Ll|
\begin{array}{ll}
u^{1/4} & \text{if } d =  1, \\
\log^q_+(u^{-1})\ u^{1/4} & \text{if } d =  2, \\
u^{1/2-\delta} & \text{if } d \ge 3.
\end{array}
\Rr.
\end{equation}
\end{thm}
\begin{thm}
\label{t:main2}
Let $p_t^\o(x,y)$ be the heat kernel associated to $L^\omega$, 
$$
\ov{p}_t(x,y) = \frac{1}{(2 \pi t)^{d/2} \sqrt{\det \ov{A}}} \exp\Ll( - \frac{1}{2t} (y-x)^\mathsf{T} \  \ov{A}^{-1} (y-x) \Rr)
$$
be the heat kernel associated to $\frac{1}{2} \nabla \cdot \ov{A} \nabla$, and $\delta > 0$. There exist constants $c > 0$ (independent of $\delta$), $q$, $C_\delta$, $\eps_\delta > 0$ such that for any $\eps > 0$, $t > 0$ satisfying $\eps/\sqrt{t}\le \eps_\delta$ and any $x \in \R^d$, one has
\begin{multline}
\label{e:main2}
\Ll|  \eps^{-d} \ \E\Ll[ p^\omega_{\eps^{-2} t}(0, \lfloor \eps^{-1} x \rfloor)\Rr] - \ov{p}_t(0,x)  \Rr| \\
\le \frac{C_\delta}{t^{d/2}} \ \Ll(\Psi_{q,\delta}\Ll(\frac{\eps^2}{t}\Rr)\Rr)^{1/(d+3)} \exp\Ll[ -c \Ll(\frac{|x|^2}{t} \wedge |\eps^{-1} x| \Rr) \Rr].
\end{multline}
In particular, for any $s > 0$, one has
\begin{equation*}
\sup_{x \in \R^d} \sup_{t \ge s} \Ll|  \eps^{-d} \ \E\Ll[ p^\omega_{\eps^{-2} t}(0, \lfloor \eps^{-1} x \rfloor)\Rr] - \ov{p}_t(0,x)  \Rr| = O\Ll( \Ll( \Psi_{q,\delta}(\eps^2)\Rr)^{1/(d+3)} \Rr)
\end{equation*}
as $\eps$ tends to $0$.
\end{thm}
\begin{rem}
For a given smooth function $f$ and a fixed $t > 0$, the r.h.s.\ of \eqref{e:main} is of the order of
\begin{equation}
\label{boundrough1}
\left|
\begin{array}{ll}
	\sqrt{\eps} & \text{if } d = 1, \\
\log^q(\eps^{-1}) \sqrt{\eps} & \text{if } d = 2, \\
{\eps}^{1-\delta'} & \text{if } d \ge 3, \\
\end{array}
\right.
\end{equation}
where $\delta' = 2 \delta > 0$ is arbitrary. Similarly, for fixed $t$ and $x$, the r.h.s.\ of \eqref{e:main2} is of the order of 
\begin{equation}
\label{boundrough2}
\left|
\begin{array}{ll}
\eps^{1/8} & \text{if } d = 1, \\
\log^{q/5}(\eps^{-1}) \ \eps^{1/10} & \text{if } d = 2, \\
{\eps}^{1/(d+3)-\delta''} & \text{if } d \ge 3, \\
\end{array}
\right.
\end{equation}
where $\delta'' = 2\delta/(d+3) > 0$ is arbitrary. 
\end{rem}
\subsection{Context}
Homogenization problems have a very long story, going back at least to \cite{maxwell,rayleigh}. Rigorous proofs of homogenization for periodic environments were obtained in the 70's (see for instance \cite{blp} or \cite[Chapter~1]{jko} for references), and for random environments with \cite{kozlov1,yuri1,papavara1,kun}. 
Classical methods used to show homogenization typically rely on a compactness argument, or on the ergodic theorem, both approaches leaving the question of the rate of convergence untouched. 

For continuous space and periodic coefficients, \cite[Corollary~2.7]{jko} uses spectral methods to show that
$$
\Ll|  \eps^{-d} \ p^\omega_{\eps^{-2} t}(0,\eps^{-1} x) - \ov{p}_t(0,x)  \Rr| \le C \ \frac{\eps}{t^{(d+1)/2}}.
$$ 
For random coefficients, available results are much less precise. For continuous space, \cite{yuri} gives an algebraic speed of convergence of $u_\eps$ to $\ov{u}$ for the elliptic problem and $d \ge 3$, without providing an explicit exponent. In \cite{cafsou}, the much more general case of fully nonlinear elliptic equations is considered, and a speed of convergence of a logarithmic type is proved.

Here, we focus on the convergence of the \emph{average} of $u_\eps$ to $\ov{u}$. This approach has been considered in \cite{conspe} for the elliptic problem. There, it is shown that the suitably rescaled Green function, once averaged over the randomness of the coefficients, differs from the Green function of the homogenized equation by a power of $\eps$. The exponent obtained is implicit, and depends on the ellipticity condition assumed on the random coefficients. 

In contrast, Theorems~\ref{t:main} and \ref{t:main2} provide explicit exponents, that depend only on the dimension. I conjecture that the correct order of decay with $\eps$ in Theorem~\ref{t:main} should be
$$
\left|
\begin{array}{ll}
	\sqrt{\eps} & \text{if } d = 1, \\
\log(\eps^{-1}) \ \eps & \text{if } d = 2, \\
{\eps}& \text{if } d \ge 3. \\
\end{array}
\right.
$$
This differs notably from what is obtained in Theorem~\ref{t:main} only when $d = 2$. On the other hand, it may well be that the assumption of high regularity on $f$ is only an artefact of the methods employed.

The fact that 
\begin{equation}
\label{localclt}
\sup_{x \in \R^d} \sup_{t \ge s} \Ll|  \eps^{-d} \  p^\omega_{\eps^{-2} t}(0, \lfloor \eps^{-1} x \rfloor) - \ov{p}_t(0,x)  \Rr| \xrightarrow[\eps \to 0]{\text{a.s.}} 0
\end{equation}
is known at least since \cite{barham}, where the much more difficult case where the random coefficients are Bernoulli random variables is considered (in this context, the heat kernel should be considered only within the unique infinite percolation cluster). Yet, for strictly positive random coefficients, this convergence does not hold if the distribution of the random coefficients is allowed to have a fat tail close to $0$ and when $d \ge 4$ \cite{bbhk,bbou}. Under the same circumstance and when $p^\omega$ is replaced by its average in \eqref{localclt}, the convergence fails to hold in any dimension \cite{fonmat} (see however \cite[Proposition 7.2]{andr} for a nice way to get around this problem).

An evaluation of the gap between the average of $u_\eps$ and $\ov{u}$ naturally calls for estimates on the size of the random fluctuations of $u_\eps$ around its average. In this direction and for the elliptic problem, \cite{connadny} obtains algebraic decay of the variance of $u_\eps$ (integrated over space). The exponent obtained is implicit, and depends on the ellipticity conditions\footnote{A.\ Gloria has announced improved estimates on this problem.}.
\subsection{Our approach}
In order to prove Theorem~\ref{t:main}, we will use the representation of $u_\eps$ as the expected value over the paths of a random walk, that we write $(X_t)_{t \ge 0}$. This random walk has inhomogeneous jump rates given by the $(\omega_e)_{e \in \mathbb{B}}$, and $L^\omega$ is its infinitesimal generator. For instance, one has
$$
u_\eps(t,0) = \EEo_0\Ll[ f\Ll( \eps X_{\eps^{-2} t} \Rr) \Rr],
$$
where we write $\PPo_0$ for the distribution of the random walk starting from $0$, and $\EEo_0$ for its associated expectation. From the PDE's perspective, this approach can be seen as the method of characteristics, except that here those characteristics are random. 

From the probabilist's perspective, the (pointwise) convergence of $u_\eps$ to $\ov{u}$ is equivalent to the claim that the random walk, after diffusive scaling, satisfies a central limit theorem. Quantitative estimates should thus follow if one can provide with rates of convergence in this central limit theorem.

In \cite{berry}, it is shown that there exist constants $C,q \ge 0$ such that for any $\xi$ of unit norm,
\begin{equation}
\label{e:berry}
\sup_{x \in \R} \  \left| \P\PPo_0 \left[\frac{\xi \cdot X_t}{\sigma(\xi)  \sqrt{t}} \le x \right] - \Phi(x) \right| \le C
\left|
\begin{array}{ll}
t^{-1/10}  & \text{if } d = 1, \\
\log_+^q(t) \ t^{-1/10} & \text{if } d = 2, \\
\log_+(t)\ t^{-1/5} & \text{if } d = 3, \\
t^{-1/5}  & \text{if } d \ge 4,
\end{array}
\right.
\end{equation}
where $\Phi$ is the cumulative distribution function of the standard Gaussian random variable, and $\sigma(\xi) = \xi \cdot \ov{A}\xi$.

This result has two important weak points: (1) the rates are far from the usual $t^{-1/2}$ one obtains for sums of i.i.d.\ random variables, and (2) the theorem only gives information about the projections of $X_t$ onto a fixed vector. We shall find ways to overcome these two problems. 

The classical approach for the proof of a central limit theorem for the random walk consists in decomposing it as the sum of a martingale plus a remainder term, and then show that the martingale converges (after scaling) to a Gaussian random variable, while the remainder term becomes negligible in the limit. 

In view of this, what should be done is clear: we should first find a quantitative estimate on how small the remainder term is, and second, show that the martingale converges rapidly to a Gaussian. This is indeed the method used in \cite{berry}. The control of the remainder term given there is satisfactory, and the problem lies with the quantitative central limit theorem for the martingale part.

This quantitative central limit theorem relies on the fact that one can have a sharp control of the variance of the quadratic variation of the martingale. It is shown that, after scaling, this variance decays like $t^{-1}$ when $d \ge 4$, which is the best possible rate. However, given such a control, the quantitative CLT (due to \cite{hb70,ha88}) used there only yields a decay of $t^{-1/5}$ in this case. 

Surprisingly, this exponent $1/5$ is best possible \cite{martingale_CLT}. To overcome this obstruction, we derive new quantitative CLT's for martingales, that will not yield a Berry-Esseen type of estimate, but rather measure 
$$
\sup_{f \in \mfk{L}} \Ll|\E\EEo_0\Ll[ f\Ll( \frac{\xi \cdot X_t}{\sigma(\xi) \sqrt{t}} \Rr) \Rr] - \int f \ \d \Phi \Rr|,
$$
where $\mfk{L}$ is a class of functions. When $\mfk{L}$ is the class of bounded $1$-Lipschitz functions, the supremum is often called the \emph{Kantorovich(-Rubinstein) distance}. We also consider $\mfk{L}$ to be the class of bounded $\mcl{C}^2$ functions that have first derivative bounded by $1$ and second derivative bounded by $k$, and call it the \emph{$k$-Kantorovich distance}. The martingale CLT's obtained hold for general square-integrable martingales, and are of independent interest.

Once equipped with these martingale CLT's, we apply them to the one-dimen\-sional projections of the random walk $(X_t)$, and for $d \ge 3$, we obtain rates approaching the i.i.d.\ rate of $t^{-1/2}$. To do so, we use estimates derived in \cite{berry}, most importantly on the variance of the quadratic variation of the martingale. These in turn are consequences of the $L^p$ boundedness of the corrector (for $d \ge 3$, and with logarithmic corrections for $d = 2$), and of a spatial decorrelation property of this corrector, proved in \cite[Theorem~2.1 and Proposition~2.1]{glotto}. 

In order to obtain Theorem~\ref{t:main}, we need to carry the information obtained on the projections of $X_t$ to $X_t$ itself, in a kind of quantitative version of the Cramér–Wold theorem. This is achieved through Fourier analysis, at the price of requiring the existence of weak derivatives of higher order. 

The key observation that enables to go from Theorem~\ref{t:main} to Theorem~\ref{t:main2} is the high regularity of the averaged heat kernel. In contrast to the true heat kernel, the averaged one has a gradient which is bounded by a constant times the gradient of $\ov{p}$, as is proved in \cite{connadny,deldeu}.

The estimates due to \cite{glotto} are the only place where the assumptions of independence and uniform ellipticity of the coefficients come into play. In particular, if it is shown that these estimates are valid for certain correlated environments, then the present results automatically extend to this context. The present results should also extend to continuous space with only minor change, as long as the estimates of \cite{glotto} remain true in this setting\footnote{A.\ Gloria and F.\ Otto have indeed announced extensions of their results to the continuous setting, and to random environments satisfying a Dobrushin-Shlosman mixing condition.}. 

\subsection{Organization of the paper}
We introduce the ($k$-)Kantorovich and Kolmogorov distances in section~\ref{s:dist}. In section~\ref{s:clt}, we consider general square-inte\-grable martingales, and derive quantitative CLT's with respect to the ($k$-)Kanto\-rovich distances. We then apply these results to projections of the random walk $X_t$ in section~\ref{s:rw}. The homogenization setting is taken up in section~\ref{s:homog}, and Theorem~\ref{t:main} is proved. Theorem~\ref{t:main2} is then derived in section~\ref{s:hk}. Finally, similar results for the homogenization of elliptic equations are presented in section~\ref{s:ellipt}.

%
%
%
%
%
%
%
%
\section{Distances between probability measures}
\label{s:dist}
\setcounter{equation}{0}
A function $f : \R^m \to \R^n$ is said to be $k$-Lipschitz if for any $x,y \in \R^m$, one has $|f(y)-f(x)| \le k |y-x|$.
Let $\nu$, $\nu'$ be probability measures on $\R$, and let $F_\nu$, $F_{\nu'}$ be their respective cumulative distribution functions. We define the \emph{Kantorovich distance} between $\nu$ and $\nu'$ as
\begin{equation}
\label{defl1}
\l_1(\nu,\nu') = \sup\Ll\{ \Ll| \int f \d \nu - \int f \d \nu'  \Rr|, \ f \text{ bounded and 1-Lip.} \Rr\},
\end{equation}
and the \emph{Kolmogorov distance} between $\nu$ and $\nu'$ as
\begin{equation}
\label{deflinfty}
\l_\infty(\nu,\nu') = \sup_{x \in \R} \Ll| F_{\nu'}(x) - F_{\nu}(x) \Rr| = \|F_{\nu'} - F_{\nu}\|_\infty.
\end{equation}
The notation for the Kantorovich distance becomes more transparent once we notice that (see for instance \cite[Theorem~1.14 and (2.48)]{villani})
\begin{equation}
\label{defeql1}
\l_1(\nu,\nu') = \int \Ll| F_{\nu'}(x) - F_{\nu}(x) \Rr| \d x = \|F_{\nu'} - F_{\nu}\|_1.
\end{equation}
As we will see below, bounds in the martingale CLT are improved when measured with the Kantorovich distance instead of the Kolmogorov distance. We now introduce weaker forms of the Kantorovich distance, for which the rates of convergence will be even better. For any $k \in [0,+\infty]$, we define the \emph{$k$-Kantorovich distance} as
\begin{equation*}
\l_{1,k}(\nu,\nu') = \sup\Ll\{ \Ll| \int f \d \nu - \int f \d \nu'  \Rr|, \ f \in \mcl{C}_b^2(\R,\R), \|f'\|_\infty \le 1, \|f''\|_\infty \le k \Rr\},
\end{equation*}
where $\mcl{C}_b^2(\R,\R)$ is the set of bounded twice continuously differentiable functions from $\R$ to $\R$. For $k \le k'$, one has $\l_{1,k}\le \l_{1,k'} \le \l_{1,\infty} = \l_1$. Note that if $f \in \mcl{C}^2_b(\R,\R)$, then
\begin{equation}
\label{compfl1k}
\Ll|\int f \d \nu - \int f \d \nu'\Rr| \le \|f'\|_\infty \ \l_{1,\|f''\|_\infty/\|f'\|_\infty}(\nu,\nu').
\end{equation}

In the sequel, if $X$ follows the distribution $\nu$ and $Y$ the distribution $\nu'$, we may write $\l_1(X,Y)$ to denote $\l_1(\nu,\nu')$, or also $\l_1(X,F_{\nu'})$ if convenient. If $X$ and $Y$ are defined on the same probability space with probability measure $P$ and associated expectation $E$, then for any $1$-Lipschitz function $f$, we have
$$
\Ll|E[f(X)] - E[f(Y)]\Rr| \le E|f(X) - f(Y)| \le E |X-Y| ,
$$ 
and hence
\begin{equation}
\label{coupling}
\l_1(X,Y) \le E |X-Y|.
\end{equation}

Similarly, if $X$ follows the distribution $\nu$ and $Y$ the distribution $\nu'$, we write $\l_{1,k}(X,Y)$, $\l_{1,k}(X,F_{\nu'})$ or $\l_{1,k}(\nu,\nu')$ as convenient.

%
%
%
%
%
%
%
%
\section{Martingale CLT}
\label{s:clt}
\setcounter{equation}{0}

For a square-integrable (cadlag) martingale $(M_t)_{t \in [0,1]}$ defined with respect to the probability measure $P$ and the (right-continuous) filtration $(\mcl{F}_t)_{t \ge 0}$, we write $(\langle M \rangle_t)_{t \in [0,1]}$ for its predictable quadratic variation, 
\begin{equation*}
\Delta M(t) = M_t - \lim_{s \to t^-} M_s,
\end{equation*}
and
$$
L_{2p} = E\Ll[ \sum_{0 \le t \le 1} |\Delta M(t)|^{2p} \Rr] .
$$
Recall that we denote by $\Phi$ the cumulative distribution function of the standard Gaussian random variable. In \cite{ha88}, the following is proved.
\begin{thm}[\cite{ha88}]
\label{t:ha88}
For any $p > 1$, there exists $\ov{C}_p$ (independent of $M$) such that 
\begin{equation}
\label{e:ha88} 
\l_\infty(M_1,\Phi) \le \ov{C}_p \Ll( L_{2p}^{1/(2p+1)} + \|\la M \ra_1 - 1\|_p^{p/(2p+1)} \Rr).
\end{equation}
\end{thm}
Our first result consists in showing that one can get sharper bounds if one replaces the Kolmogorov distance by the ($k$-)Kantorovich distance in \eqref{e:ha88}.
\begin{thm}
\label{t:clt}
For any $p > 1$, there exists $C_p$ (independent of $M$) such that 
\begin{equation}
\label{e:clt} 
\l_1(M_1,\Phi) \le C_p  L_{2p}^{1/(2p+1)} + 2 \|\la M \ra_1 - 1\|_1^{1/2},
\end{equation}
and for any $k \ge 0$,
\begin{equation}
\label{e:clt2} 
\l_{1,k}(M_1,\Phi) \le C_p  L_{2p}^{1/(2p+1)} + \frac{k}{2} L_{2p}^{1/p} + (k\vee 1) \|\la M \ra_1 - 1\|_1.
\end{equation}
\end{thm}
\begin{rem}
Naturally, one has $\|\la M \ra_1 - 1\|_1 \le \|\la M \ra_1 - 1\|_p$, and the statements are only interesting when this quantity, and also $L_{2p}$, are small, so Theorem~\ref{t:clt} indeed provides better rates of convergence than Theorem~\ref{t:ha88}. It is shown in \cite{martingale_CLT} that it is not possible to change the exponent $p/(2p+1)$ appearing on the term $\|\la M \ra_1 - 1\|_p$ in the r.h.s.\ of \eqref{e:ha88} by any higher exponent. It would be interesting to investigate how sharp \eqref{e:clt} is in this respect. The term $\|\la M \ra_1 - 1\|_1$ appearing on the r.h.s.\ of \eqref{e:clt2} cannot be improved. Indeed, let $(B_s)_{s \ge 0}$ be a standard Brownian motion, and consider the martingale $M_s = B_{(1+\eps) s}$. Since the martingale is continuous, $L_{2p}$ vanishes, while one has $\|\la M \ra_1 - 1\|_1 = \eps$. On the other hand, the cosine function has first and second derivatives bounded by $1$, and thus
$$
\l_{1,1}(M_1,\Phi) \ge E[\cos(B_1)] - E[\cos(M_1)] = e^{-1/2} - e^{-(1+\eps)/2} \sim \frac{\eps}{2} \qquad (\eps \to 0),
$$
thus justifying optimality of the exponent on $\|\la M \ra_1 - 1\|_1$.
\end{rem}
\begin{rem}
A quantitative martingale CLT expressed in terms of the Kantorovich distance was already formulated in \cite[Theorem~8.1.16]{raru}. Terms involved in the bound are however difficult to estimate in practical situations, in contrast to what is obtained in Theorem~\ref{t:clt}.
\end{rem}

In order to prove Theorem~\ref{t:clt}, we will rely on the following non-uniform version of Theorem~\ref{t:ha88}.
\begin{thm}[\cite{hj,ha88}]
\label{t:hj}
For any $p > 1$, there exists $\td{C}_p$ (independent of $M$) such that if $L_{2p} + \|\la M \ra_1 - 1\|_p^{p} \le 1$, then for any $x \in \R$,
$$
\Ll| P[M_1 \le x] - \Phi(x) \Rr| \le \frac{\td{C}_p}{1+|x|^{2p}} \Ll( L_{2p}^{1/(2p+1)} + \|\la M \ra_1 - 1\|_p^{p/(2p+1)} \Rr).
$$
\end{thm}
\cite[Theorem~1]{hj} is the equivalent statement concerning discrete time martingales. Theorem~\ref{t:hj} can be derived from its discrete time version by applying the approximation procedure explained in \cite[Section~4]{ha88} (in \cite{ha88}, \emph{locally} square-integrable martingales are considered, while we stick here with plainly square-integrable martingales. There is no loss of generality however, since a locally square-integrable martingale is in fact a square-integrable one if $\|\la M \ra_1 - 1\|_1$ is to be finite. One can thus skip the localization procedure at the end of \cite[Section~4]{ha88}). 

\begin{proof}[Proof of Theorem~\ref{t:clt}] We start by proving that there exists $C_p$ (independent of $M$) such that \eqref{e:clt} holds. We decompose the proof of this into three steps.

\medskip

\noindent \emph{Step I.1.} 
We first prove the claim assuming that 
\begin{equation}
\label{quadvar1}
\la M \ra_1 = 1 \text{ a.s.}
\end{equation}
and that $L_{2p} \le 1$.
Under this condition, Theorem~\ref{t:hj} ensures that
$$
\Ll| P[M_1 \le x] - \Phi(x) \Rr| \le \frac{\td{C}_p}{1+|x|^{2p}} L_{2p}^{1/(2p+1)}.
$$
We thus have, after possibly enlarging $\td{C}_p$,
$$
\l_1(M_1,\Phi) = \int \Ll| P[M_1 \le x] - \Phi(x) \Rr| \d x \le \td{C}_p L_{2p}^{1/(2p+1)},
$$
which is the desired result.

\medskip

\noindent \emph{Step I.2.} 
We now no longer impose that condition \eqref{quadvar1} holds, but keep with the assumption that $L_{2p} \le 1$. Following an idea probably due to \cite{dvo}, we introduce
$$
\tau = \sup \{t \le 1 : \la M \ra_t \le 1\}.
$$
Note that $\tau$ is a stopping time, since
$$
\{ \tau \le t \} = \bigcap_{\eps > 0} \{ \la M \ra_{t+\eps} > 1 \} \in \mcl{F}_{t}.
$$
We define
$$
\la M \ra_{\tau^-} = 
\left|
\begin{array}{ll}
\la M \ra_1 & \text{if } \la M \ra_1 \le 1, \\
\lim_{t \to \tau^-} \la M \ra_t  & \text{otherwise},
\end{array}
\right.
$$
and
$$
M_{\tau^-} = 
\left|
\begin{array}{ll}
M_1 & \text{if } \la M \ra_1 \le 1, \\
\lim_{t \to \tau^-} M_t  & \text{otherwise}.
\end{array}
\right.
$$
Note that $\la M \ra_{\tau^-} \le 1$. Let $(B_s)_{s \ge 0}$ be a standard Brownian motion, independent of the martingale. We define
$$
\td{M}_s = 
\left|
\begin{array}{ll}
M_s & \text{if } 0 \le s < \tau ,\\
M_{\tau^-} & \text{if } \tau \le s \le 1 ,\\
M_{\tau^-} + B_{s-1} & \text{if } 1 \le s \le 2-\la M \ra_{\tau^-}  ,\\
M_{\tau^-} + B_{1-\la M \ra_{\tau^-}} & \text{if }2-\la M \ra_{\tau^-} \le s \le 2 .\\
\end{array}
\right.
$$
By construction, $\td{M}$ is a martingale, and
$$
\la \td{M} \ra_2 - \la \td{M} \ra_1 = 1-\la {M} \ra_{\tau^-},
$$
hence $\la \td{M} \ra_2 = 1$. Naturally, the fact that $\td{M}$ is defined on $[0,2]$ instead of $[0,1]$ plays no role, and this martingale satisfies condition \eqref{quadvar1} (at time $2$). Let us write
$$
\td{L}_{2p} = E\Ll[ \sum_{0 \le t \le 2} |\Delta \td{M}(t)|^{2p} \Rr] \le L_{2p} \le 1.
$$
We learn from the the first step of the proof that
\begin{equation}
\label{oneside}
\l_1(\td{M}_2,\Phi) \le \td{C}_p \td{L}_{2p}^{1/(2p+1)} \le \td{C}_p {L}_{2p}^{1/(2p+1)}.
\end{equation}
We now want to use the fact that 
\begin{equation}
\label{triangle}
\l_1(M_1,\Phi) \le \l_1(M_1,\td{M}_2) + \l_1(\td{M}_2,\Phi)
\end{equation}
to estimate $\l_1(M_1,\Phi)$. In view of \eqref{coupling}, we have
$$
\l_1(M_1,\td{M}_2) \le E\Ll[ \Ll|M_1 - \td{M}_2\Rr| \Rr].
$$
Note that
$$
M_1 - \td{M}_2 = M_1 - M_\tau + \Delta M(\tau) \1_{\la M \ra_1 > 1} - B_{1-\la M \ra_{\tau^-}},
$$
and thus
$$
E\Ll[ \Ll|M_1 - \td{M}_2\Rr| \Rr] \le E\Ll[ \Ll|M_1 - M_\tau\Rr| \Rr] + E\Ll[ \Ll|\Delta M(\tau)\Rr| \Rr] + E\Ll[ \Ll|B_{1-\la M \ra_{\tau^-}}\Rr| \Rr].
$$
Let us write $a_1 + a_2 + a_3$ for the latter sum, with obvious identifications.
We bound the contribution of each of these terms successively.
$$
a_1  \le E\Ll[ (M_1 - M_\tau)^2 \Rr]^{1/2} = E\Ll[ \la M \ra_1 - \la M\ra_\tau \Rr]^{1/2},
$$
since $\tau \le 1$ is a stopping time. Now, either $\tau = 1$, in which case $\la M \ra_1 - \la M\ra_\tau  = 0$, or $\tau < 1$, in which case $\la M \ra_\tau \ge 1$. In both cases, we have 
$$
\la M \ra_1 - \la M\ra_\tau \le \Ll| \la M \ra_1 - 1 \Rr|,
$$
and thus
$$
a_1 \le \|\la M \ra_1 - 1 \|_1^{1/2}.
$$
As for $a_2$, we have
\begin{equation}
\label{deltam}
a_2 = E\Ll[ \Ll|\Delta M(\tau)\Rr| \Rr] \le E\Ll[ \Ll|\Delta M(\tau)\Rr|^{2p} \Rr]^{1/(2p)} \le L_{2p}^{1/(2p)}.
\end{equation}
For the third term, we have
$$
a_3 = c \ E\Ll[\Ll| 1-\la M\ra_{\tau^-}\Rr|^{1/2}\Rr] \le c \ E\Ll[\Ll| 1-\la M\ra_{\tau^-}\Rr|\Rr]^{1/2},
$$
where $c = E[|B_1|] \le 1$. We decompose the last expectation as
$$
E\Ll[\Ll| 1-\la M\ra_{\tau^-}\Rr| \ \1_{\la M \ra_1 \le 1}\Rr] + E\Ll[\Ll| 1-\la M\ra_{\tau^-}\Rr| \ \1_{\la M \ra_1 > 1}\Rr].
$$
The first term is bounded by $\|1-\la M \ra_1\|_1$, while the second term is smaller than
$$
E\Ll[ \Delta \la M \ra(\tau) \Rr] \le E\Ll[(\Delta M(\tau))^2\Rr]
$$
(to see this, consult for instance the proof of \cite[Theorem~4.2]{js}). The latter is bounded by 
$$
E\Ll[(\Delta M(\tau))^{2p}\Rr]^{1/p} \le L_{2p}^{1/p}.
$$
To sum up, we have shown that
\begin{eqnarray}
\label{m1m2}
\l_1(M_1,\td{M}_2) & \le & \|\la M \ra_1 - 1 \|_1^{1/2} + L_{2p}^{1/(2p)} + \sqrt{\|1-\la M \ra_1\|_1 + L_{2p}^{1/p}} \notag \\
& \le & 2 \|\la M \ra_1 - 1 \|_1^{1/2} + 2 L_{2p}^{1/(2p)}.
\end{eqnarray}
Since we assume that $L_{2p} \le 1$, we have $L_{2p}^{1/(2p)} \le L_{2p}^{1/(2p+1)}$, and equations \eqref{triangle}, \eqref{oneside} and \eqref{m1m2} give us that 
\begin{equation*}
\l_1(M_1, \Phi)  \le  (\td{C}_p+2)  L_{2p}^{1/(2p+1)} + 2 \|\la M \ra_1 - 1 \|_1^{1/2}
,
\end{equation*}
which is what we wanted to prove.
\medskip

\noindent \emph{Step I.3.} There remains to consider the case when $L_{2p} > 1$. It follows from \eqref{coupling} that
$$
\l_1(M_1,\Phi) \le c + \|M_1\|_1,
$$
where $c$ is the $L_1$ norm of a standard Gaussian, $c \le 1$. Moreover,
$$
\|M_1\|_1 \le \|M_1\|_2 = \|\la M \ra_1\|_1^{1/2} \le \Ll(1+\|\la M \ra_1-1\|_1\Rr)^{1/2}.
$$
As a consequence, it is always true that
$$
\l_1(M_1,\Phi) \le 2+\|\la M \ra_1 - 1\|_1^{1/2}.
$$
The theorem is thus clearly true when $L_{2p} > 1$ as soon as $C_p \ge 2$, and this finishes the proof of \eqref{e:clt2}.

\medskip

We now proceed to show that there exists $C_{p}$ (independent of $M$ and $k$) such that \eqref{e:clt2} holds, and decompose the proof of this fact into two steps.

\medskip

\noindent \emph{Step II.1.} We assume first that $L_{2p} \le 1$, and consider again the martingale $\td{M}$ as constructed in step I.2. Since $\langle \td{M} \rangle_2 = 1$, we know from step I.1 that
\begin{equation}
\label{stepii1}
\l_1(\td{M}_2,\Phi) \le \td{C}_p L_{2p}^{1/(2p+1)}.
\end{equation}
Let 
$$
\ov{M}_2 = M_\tau + B_{1-\langle M \rangle_{\tau^-}},
$$
and observe that
$$
\td{M}_2 = \ov{M}_2 - \Delta M(\tau) \1_{\langle M \rangle_1 > 1}.
$$
We have
$$
\l_1(\ov{M}_2,\Phi) \le \l_1(\ov{M}_2,\td{M}_2) + \l_1(\td{M}_2,\Phi).
$$
The first term on the r.h.s.\ is smaller than $E[|\Delta M (\tau)|]$, and we have seen in \eqref{deltam} that this is smaller than $L_{2p}^{1/(2p)} \le L_{2p}^{1/(2p+1)}$. Using also \eqref{stepii1}, we obtain
\begin{equation}
\label{ovmphi}
\l_{1,k}(\ov{M}_2,\Phi) \le \l_1(\ov{M}_2,\Phi) \le (\td{C}_p+1) L_{2p}^{1/(2p+1)}.
\end{equation}
Let $f \in \mcl{C}_b^2 (\R,\R)$ be such that $\|f'\|_\infty \le 1$ and $\|f''\|_\infty \le k$. We will show that
\begin{equation}
\label{mtauovm2}
\Ll| E[f(\ov{M}_2)] - E[f(M_\tau)]   \Rr| \le \frac{k}{2} \Ll(L_{2p}^{1/p} + \|\langle M \rangle_1 - 1 \|_1 \Rr).
\end{equation}
Indeed, since $f \in \mcl{C}_b^2 (\R,\R)$ and $\|f''\|_\infty \le k$, we have 
$$
\Ll|E\Ll[ f(\ov{M}_2) - f(M_\tau) - (\ov{M}_2 - M_\tau) f'(M_\tau) \Rr] \Rr| \le \frac{k}{2} E\Ll[(\ov{M}_2 - M_\tau)^2\Rr].
$$
But
\begin{eqnarray*}
E\Ll[(\ov{M}_2 - M_\tau) f'(M_\tau) \Rr] & = & E[B_{1-\langle M \rangle_{\tau^-}}f'(M_\tau)] \\
& = & E\big[E[B_{1-\langle M \rangle_{\tau^-}} \ | \ \mcl{F}_{\tau}] \ f'(M_\tau)\big],
\end{eqnarray*}
and $E[B_{1-\langle M \rangle_{\tau^-}} \ | \ \mcl{F}_{\tau}] = 0$ since $B$ and $M$ are independent. On the other hand,
$$
E\Ll[(\ov{M}_2 - M_\tau)^2\Rr] = E[(B_{1-\langle M \rangle_{\tau^-}})^2] = E[1-\langle M \rangle_{\tau^-}],
$$
and we have seen in step I.2, while treating the term $a_3$, that
$$
E[1-\langle M \rangle_{\tau^-}] \le \|\langle M \rangle_1 - 1 \|_1 + L_{2p}^{1/p}.
$$
As a consequence, \eqref{mtauovm2} is proved, and thus
\begin{equation}
\label{mtauovm2bis}
\l_{1,k}(M_\tau,\ov{M}_2) \le \frac{k}{2} \Ll(L_{2p}^{1/p} + \|\langle M \rangle_1 - 1 \|_1 \Rr).
\end{equation}
We now show that 
\begin{equation}
\label{mtaum1}
\l_{1,k}(M_\tau,M_1) \le \frac{k}{2} \|\langle M \rangle_1 - 1 \|_1 	
\end{equation}
using the same technique. We write
$$
\Ll|E\Ll[ f(M_1) - f(M_\tau) - (M_1 - M_\tau) f'(M_\tau) \Rr] \Rr| \le \frac{k}{2} E\Ll[(M_1 - M_\tau)^2\Rr],
$$
and observe that
\begin{equation*}
E\Ll[({M}_1 - M_\tau) f'(M_\tau) \Rr]  =  E\big[E[({M}_1 - M_\tau) \ | \ \mcl{F}_{\tau}] \ f'(M_\tau)\big] = 0,
\end{equation*}
since $M$ is a martingale and $\tau$ a stopping time. On the other hand, we have seen while treating the term $a_1$ in step I.2\ that
$$
E\Ll[(M_1 - M_\tau)^2\Rr] \le \|\langle M \rangle_1 - 1 \|_1,
$$
and thus \eqref{mtaum1} is proved. Combining \eqref{ovmphi}, \eqref{mtauovm2bis} and \eqref{mtaum1}, we thus obtain
$$
\l_{1,k}(M_1,\Phi) \le \Ll(\td{C}_p + 1\Rr)L_{2p}^{1/(2p+1)} + \frac{k}{2} L_{2p}^{1/p} + k \|\langle M \rangle_1 - 1 \|_1,
$$
and this proves \eqref{e:clt2} for $L_{2p} \le 1$.

\medskip

\noindent \emph{Step II.2.} We now conclude by considering the case when $L_{2p} > 1$. We learn from step I.3\ that
$$
\l_{1,k}(M_1,\Phi) \le 2+\|\la M \ra_1 - 1\|_1^{1/2}.
$$
Since for any $x \ge 0$, we have $\sqrt{x} \le 1 + x/2$, we thus obtain
$$
\l_{1,k}(M_1,\Phi) \le 3 + \frac{1}{2} \|\la M \ra_1 - 1\|_1,
$$
and thus relation \eqref{e:clt2} holds when $L_{2p} > 1$, provided we choose $C_{p} \ge 3$.
\end{proof}
%
%
%
%
%
%
%
%
\section{The random walk in random conductances}
\label{s:rw}
\setcounter{equation}{0}
Let $0 < \alpha \le \beta < + \infty$, and $\Omega = [\alpha,\beta]^\mathbb{B}$. For any family $\omega = (\omega_e)_{e \in \mathbb{B}} \in \Omega$, we consider the Markov process $(X_t)_{t \ge 0}$ whose jump rate between $x$ and a neighbour $y$ is given by $\omega_{x,y}$. We write $\PPo_x$ for the law of this process starting from $x \in \Z^d$, $\EEo_x$ for its associated expectation. Its infinitesimal generator is $L^\omega$ defined in \eqref{defLom}. We assume that the $(\omega_e)_{e \in \mathbb{B}}$ are themselves i.i.d.\ random variables under the measure $\P$ (with associated expectation~$\E$). We write $\ov{\P} = \P\PPo_0$ for the \emph{annealed} measure. It was shown in \cite{kipvar} that under $\ov{\P}$ and as $\eps$ tends to $0$, the process $\sqrt{\eps} X_{\eps^{-1} t}$ converges to a Brownian motion, whose covariance matrix we write $\ov{A}$ (in \cite{sid}, it is shown that under our present assumption of uniform ellipticity, the invariance principle holds under $\PPo_0$ for almost every $\omega$).

Let $\xi \in \R^d$ be a vector of unit $L^2$ norm. The purpose of this section is to give sharp estimates on the $k$-Kantorovich distance between $\xi \cdot X_t/\sqrt{t}$ and $\Phi_{\sigma(\xi)}$, where we write $\Phi_\sigma$ to denote the cumulative distribution function of a Gaussian random variable with variance $\sigma^2$, and $\sigma(\xi) = \xi \cdot \ov{A} \xi$.

\begin{thm}
\label{t:rw}
For any $\delta > 0$, there exists a constant $C$ (which may depend on the dimension) such that for any $k \ge 0$ and any $\xi$ of unit norm, one has
\begin{equation}
\label{e:rw}
\l_{1,k}\left(\frac{\xi \cdot X_t}{\sqrt{t}}, \Phi_{\sigma(\xi)}\right) \le C \ (k \vee 1) \ \Psi_{q,\delta}(t^{-1})
\end{equation}
for some $q \ge 0$, where in the l.h.s.,\ $\xi \cdot X_t/\sqrt{t}$ stands for the distribution of this random variable under the measure $\ov{\P}$, and where $\Psi_{q,\delta}$ was defined in \eqref{defpsi}.
\end{thm}
\begin{rem}
When $d \ge 3$, the exponent of decay in \eqref{e:rw} can thus be made arbitrarily close to $1/2$, and this is the exponent one gets when considering sums of i.i.d.\ random variables with finite third moment.
\end{rem}
\begin{rem}
By the same reasoning, one can also prove that there exist constants $C$ (which may depend on the dimension) and $q$ such that, for any $\xi$ of unit norm, one has
$$
\l_1\left(\frac{\xi \cdot X_t}{\sqrt{t}}, \Phi_{\sigma(\xi)}\right) \le C
\left|
\begin{array}{ll}
 t^{-1/8} & \text{if } d = 1, \\
 \log_+^q(t) \ t^{-1/8}  & \text{if } d = 2, \\
 \log_+^{1/4}(t) \ t^{-1/4}  & \text{if } d = 3, \\
 t^{-1/4}   & \text{if } d \ge 4,
\end{array}
\right.
$$
where again $\xi \cdot X_t/\sqrt{t}$ stands for the distribution of this random variable under the measure $\ov{\P}$.
\end{rem}
\begin{proof}[Proof of Theorem~\ref{t:rw}]
The argument follows \cite{berry} closely. We first treat the case $d \ge 2$. For $\mu > 0$, we decompose $\xi \cdot X_t$ as $M_\mu(t) + R_\mu(t)$, where $M_\mu(t)$ and $R_\mu(t)$ are defined in \cite[(3.7)-(3.8)]{berry} respectively. Under the measure $\ov{\P}$ (and for the natural filtration associated to $X$), $M_\mu$ is a martingale with stationary increments, and we write $\sigma_\mu = \ov{\E}[M_\mu(1)^2]$. We have
\begin{multline*}
\l_{1,k}\left(\frac{\xi \cdot X_t}{\sqrt{t}}, \Phi_{\sigma(\xi)}\right) \\
\le \l_{1,k}\left(\frac{\xi \cdot X_t}{\sqrt{t}}, \frac{M_\mu(t)}{\sqrt{t}}\right) + \l_{1,k}\left(\frac{M_\mu(t)}{\sqrt{t}}, \Phi_{\sigma_\mu}\right) + \l_{1,k}\left(\Phi_{\sigma_\mu}, \Phi_{\sigma(\xi)}\right),
\end{multline*}
with the understanding that random variables stand in place of their respective distributions under the measure $\ov{\P}$. Let us write the three terms in the r.h.s.\ above as $b_1 + b_2 + b_3$, and proceed to evaluate each of these terms for the specific choice $\mu = 1/t$. Considering \eqref{coupling}, we can bound the term $b_1$ by
$$
\l_{1}\left(\frac{\xi \cdot X_t}{\sqrt{t}}, \frac{M_{1/t}(t)}{\sqrt{t}}\right) \le \ov{\E}\Ll[\frac{|R_{1/t}(t)|}{\sqrt{t}}\Rr] \le \frac{\ov{\E}[(R_{1/t}(t))^2]^{1/2}}{\sqrt{t}},
$$
and \cite[Proposition~3.4]{berry} gives us that
$$
\ov{\E}[(R_{1/t}(t))^2] \le
C \left|
\begin{array}{ll}
\log_+^q(t) & \text{if } d = 2, \\
1 & \text{if } d \ge 3, 
\end{array} 
\right.
$$
for some constants $C$ and $q$.

To handle the term $b_3$, consider a standard Gaussian random variable $\mcl{N}$. Then $\sigma \mcl{N}$ has $\Phi_{\sigma}$ as its cumulative ditribution function, hence 
$$
\l_{1,k}\left(\Phi_{\sigma}, \Phi_{\sigma'}\right) \le \l_1\left(\Phi_{\sigma}, \Phi_{\sigma'}\right) \le E[|\sigma \mcl{N} - \sigma' \mcl{N}|] = E[|\mcl{N}|] \ |\sigma - \sigma'|.
$$
Since $E[|\mcl{N}|] \le 1$, the term $b_3$ is bounded by $|\sigma_{1/t} - \sigma(\xi)|$. It is shown in \cite[Theorem~1]{glotto} (and recalled in \cite[Proposition~3.3]{berry}) that
$$
|\sigma_{1/t} - \sigma(\xi)| \le C 
\left|
\begin{array}{ll}
\log_+^q(t) \ t^{-1} & \text{if } d = 2, \\
t^{-3/2} & \text{if } d = 3, \\
\log_+(t) \ t^{-2} & \text{if } d = 4, \\
t^{-2}& \text{if } d \ge 5,
\end{array}
\right.
$$
which is much better than what we need for our purpose. 

We finally turn to the term $b_2$. For any $p > 1$, we introduce 
$$
L_{2p}(t) = \frac{1}{t^{p}} \ \ov{\E}\Ll[ \sum_{0 \le s \le t} |\Delta M_{1/t}(t)|^{2p}  \Rr].
$$
Theorem~\ref{t:clt} tells us that if $L_{2p}(t) \le 1$, then $b_2$ is smaller than
$$
\Ll(C_p + \frac{k}{2}\Rr) (L_{2p}(t))^{1/(2p+1)} + (k \vee 1) \Ll\|\frac{\la M_{1/t}\ra_t}{t} - \sigma_{1/t}\Rr\|_1^{1/2}.
$$
Replacing the exponent $4$ by $2p$ leaves the proof of \cite[(3.11)]{berry} unchanged, and ensures that
$$
L_{2p}(t)  \le C 
\left|
\begin{array}{ll}
	\log_+^q(t) \ t^{-p+1} & \text{if } d = 2, \\
	t^{-p+1} & \text{if } d \ge 3,
\end{array}
\right.
$$
for some constants $C$ and $q$ depending on $p$. In particular, it is always true that $L_{2p}(t)$ tends to $0$ as $t$ tends to infinity. We fix $p$ large enough so that
\begin{equation}
\label{defp}
\frac{p-1}{2p+1} > \frac{1}{2} - \delta.
\end{equation}
With such a choice for $p$, we have $(L_{2p}(t))^{1/(2p+1)} = o(t^{\delta -1/2})$. 

On the other hand, \cite[(3.10)]{berry} ensures that
$$
\Ll\|\frac{\la M_{1/t}\ra_t}{t} - \sigma_{1/t}\Rr\|_2^2 \le C
\left|
\begin{array}{ll}
	\log_+^q(t) \ t^{-1/2} & \text{if } d = 2, \\
	\log_+(t) \ t^{-1} & \text{if } d = 3, \\
t^{-1} & \text{if } d \ge 4.
\end{array}
\right.
$$
Since 
$$
\Ll\|\frac{\la M_{1/t}\ra_t}{t} - \sigma_{1/t}\Rr\|_1 \le \Ll\|\frac{\la M_{1/t}\ra_t}{t} - \sigma_{1/t}\Rr\|_2,
$$
this finishes the proof of Theorem~\ref{t:rw} for $d\ge 2$ and $t$ large enough, and it is easy to see that the l.h.s.\ of \eqref{e:rw} is bounded for smaller $t$. The one-dimensional case is obtained in a similar way, following \cite[Section~9]{berry}.
\end{proof}
%
%
%
%
%
%
%
%
\section{Homogenization}
\label{s:homog}
\setcounter{equation}{0}

We consider the discrete parabolic equation with random coefficients
\begin{equation*}
\label{parabeq}
\tag{DPE$^\omega$}
\left\{
\begin{array}{ll}
\displaystyle{\frac{\dr u}{\dr t}  = L^\omega u    } & \text{on } \R_+ \times \Z^d, \\
\\
\displaystyle{   u(0,\cdot) = f   } & \text{on } \Z^d,
\end{array}
\right.	
\end{equation*}
where $f: \Z^d \to \R$, $L^\omega$ is the operator defined in \eqref{defLom}, and by $L^\omega u(t,x)$, we understand $L^\omega u(t,\cdot) (x)$. Note that $L^\omega$ is the discrete analog of a divergence form operator. 

For a fixed $\omega \in \Omega$, we say that $u$ is a \emph{solution of} \eqref{parabeq} if it is continuous on $[0,+\infty) \times \Z^d$, has continuous time derivative there (in other words, $u(\cdot,x)$ is in $\mcl{C}^1(\R_+,\R)$ for every $x \in \Z^d$), and satisfies the identities dislayed in \eqref{parabeq}.

\begin{prop}
\label{discexistunique}
For any $\omega \in \Omega$ and any bounded initial condition $f$, there exists a unique bounded solution $u$ of \eqref{parabeq}, and it is given by 
\begin{equation}
\label{urw}
u(t,x) = \EEo_x[f(X_t)].
\end{equation}
\end{prop}
This is a very well known result. Checking that \eqref{urw} is indeed a solution is a direct consequence of the definition of the Markov chain. To see uniqueness, take $\td{u}$ a bounded solution of \eqref{parabeq}. Letting $\td{M}_s = \td{u}(t-s,X_s)$, one can show that $(\td{M}_s)_{0 \le s \le t}$ is a martingale under $\PPo_x$ for any $x \in \Z^d$, and as a consequence,
$$
\td{u}(t,x) = \EEo_x[\td{M}_0] = \EEo_x[\td{M}_t] = \EEo_x[\td{u}(0,X_t)] = \EEo_x[f(X_t)],
$$
which is the function defined in \eqref{urw}.

For a symmetric positive-definite matrix $\ov{A}$, we consider the equation \eqref{cparabeq} given in the introduction. We say that $\ov{u}$ is a \emph{solution of} \eqref{cparabeq} if it is continuous on $\R_+ \times \R^d$, has a continuous first derivative in the time variable and continuous first and second derivatives in the space variable on $(0,+\infty) \times \R^d$, and satisfies the identities dislayed in \eqref{cparabeq}.

\begin{prop}
\label{contexistunique}
For any bounded continuous initial condition $f$, there exists a unique bounded solution $\ov{u}$ of \eqref{cparabeq}, and it is given by 
\begin{equation}
\label{curw}
\ov{u}(t,x) = \EE_x[f(B_t)],
\end{equation}
where, under the measure $\PP_x$, $B_t$ is a Brownian motion with covariance matrix $\ov{A}$ that starts at $x$.
\end{prop}
 Again, this result is standard. It is proved in the same way as Proposition~\ref{discexistunique}, with the help of Itô's formula. 

\begin{rem}
The boundedness assumption in Propositions~\ref{discexistunique} and \ref{contexistunique} could be changed for being subexponential. More precisely, let $f : \Z^d \to \R$ be such that for any $\alpha > 0$, $|f(x)| = O(e^{\alpha |x|})$. Then there exists a unique solution $u$ of \eqref{parabeq} such that, for any $\alpha > 0$ and any $t \ge 0$, $\sup_{s \le t} |u(s,x)| = O(e^{\alpha |x|})$. The boundedness condition was chosen merely for convenience.
\end{rem}

We now define rescaled solutions of the parabolic equation with random coefficients. For a bounded continuous function $f:\R^d \to \R$, we let $u^{(\eps)}$ be the bounded solution of \eqref{parabeq} with initial condition the function $x \mapsto f(\eps  x)$, and for any $t\ge 0$ and $x \in \R^d$, we let
\begin{equation}
\label{uepsprob}
u_\eps(t,x) = u^{(\eps)}(\eps^{-2} t, \lfloor \eps^{-1}x \rfloor) = \EEo_{\lfloor \eps^{-1} x \rfloor} [f(\eps X_{\eps^{-2} t})].
\end{equation}
It is well understood (see for instance \cite[Chapter~3]{blp}) that the probabilistic approach yields pointwise convergence of $u_\eps$ to the solution of the homogenized problem. The following result is folklore (see also \cite{lejay} where the homogenization of random operators in continuous space is obtained using the probabilistic approach).
\begin{thm}
\label{t:homog}
There exists a symmetric positive-definite matrix $\ov{A}$ (independent of $f$) such that for every $t \ge 0$ and $x \in \R^d$, we have
\begin{equation}
\label{pointwise}
u_\eps(t,x) \xrightarrow[\eps \to 0]{\text{(prob.)}} \ov{u}(t,x),
\end{equation}
where $\ov{u}$ is the bounded solution of \eqref{cparabeq} with initial condition $f$.
\end{thm}
\begin{proof}
Let us write $(\theta_x)$ to denote the translations on $\Omega$, acting according to $(\theta_x \ \omega)_{y,z} = \omega_{x+y,x+z}$. The distribution of $X$ under $\PPo_x$ is the same as the one of $X+x$ under $\PP^{\theta_x  \omega}_0$ (both are Markov processes with the same initial condition and the same transition rates). Using this observation in \eqref{uepsprob}, we obtain that
$$
u_\eps(t,x) = \EE^{\theta_{\lfloor \eps^{-1} x \rfloor}  \omega}_0 [f(\eps X_{\eps^{-2} t} + x_\eps)],
$$
where $x_\eps = \eps \lfloor \eps^{-1} x \rfloor$.

Since the measure $\P$ is invariant under translations, $u_\eps(t,x)$ has the same distribution as
\begin{equation}
\label{testf}
\EEo_0 [f(\eps X_{\eps^{-2} t} + x_\eps)].
\end{equation}
It is proved in \cite{kipvar,masi} that for some symmetric positive-definite $\ov{A}$ (independent of $f$), the quantity in \eqref{testf} converges in probability to $\EE_0[f(B_t +x)]$ as $\eps$ tends to $0$, where $B$ is a Brownian motion with covariance matrix $\ov{A}$.
\end{proof}
\begin{rem}
It would be interesting to replace the convergence in probability in \eqref{pointwise} by an almost sure convergence. Note that almost sure convergence for $x = 0$ is equivalent to an almost sure central limit theorem for the random walk, and this is proved in \cite{sid}. Theorem~\ref{t:homog} contrasts with for instance \cite[Theorem~7.4]{jko}, where weak convergence of an analogue of $u_\eps$ is proved, but for almost every environment.
\end{rem}
We start the proof of Theorem~\ref{t:main} with two lemmas with a Fourier-analytic flavour.
\begin{lem}
\label{l:fourier}
Let $Z$ be a random variable following the distribution $\nu$, $\mcl{N}$ be a standard $d$-dimensional Gaussian random variable independent of $Z$, and $\sigma > 0$. We have
$$
E[f(Z + \sigma N)] = (2\pi)^{-d} \int_{\R^d} \exp\Ll(-\frac{\sigma^2 |\xi|^2}{2}\Rr) \hat{f}(\xi) \hat{\nu}(\xi) \ \d \xi,
$$
where
$$
\hat{f}(\xi) = \int e^{i \xi \cdot x} f(x) \ \d x,
$$
and 
$$
\hat{\nu}(\xi) = \int e^{- i \xi \cdot x}  \ \d \nu(x).
$$
\end{lem}
\begin{proof}
Let us write 
$$
g_\sigma(x) = \frac{1}{(2\pi \sigma^2)^{d/2}} \exp\Ll( - \frac{|x|^2}{2 \sigma^2} \Rr).
$$
Note first that
\begin{equation}
\label{Fourierg}
\hat{g}_{1/\sigma}(x) = \exp\Ll( - \frac{|x|^2}{2 \sigma^2} \Rr) = (2\pi \sigma^2)^{d/2} g_\sigma(x).
\end{equation}
The distribution of $Z + \sigma \mcl{N}$ has a density (w.r.\ to Lebesgue measure) at point $z$ which is given by
\begin{eqnarray*}
\int g_\sigma(z-x) \ \d \nu(x) & \stackrel{\text{\eqref{Fourierg}}}{=} & 
{(2\pi \sigma^2)^{-d/2}} \int \hat{g}_{1/\sigma}(z-x) \ \d \nu(x) \\
& = & {(2\pi \sigma^2)^{-d/2}} \int e^{i \xi \cdot (z-x)} g_{1/\sigma}(\xi) \ \d \xi \ \d \nu(x) \\
& = & {(2\pi \sigma^2)^{-d/2}} \int e^{i \xi \cdot z} g_{1/\sigma}(\xi) \hat{\nu}(\xi) \ \d \xi.
\end{eqnarray*}
As a consequence (and using the fact that $f$ and $\hat{\nu}$ are bounded), we have
\begin{eqnarray*}
E[f(Z + \sigma N)] & = & {(2\pi \sigma^2)^{-d/2}} \int f(z) e^{i \xi \cdot z} g_{1/\sigma}(\xi) \hat{\nu}(\xi) \ \d \xi \ \d z\\
& = & {(2\pi \sigma^2)^{-d/2}} \int g_{1/\sigma}(\xi) \hat{f}(\xi) \hat{\nu}(\xi) \ \d \xi.
\end{eqnarray*}
Since 
$$
{(2\pi \sigma^2)^{-d/2}}  g_{1/\sigma}(\xi) = (2\pi)^{-d} \exp\Ll( - \frac{\sigma^2 |x|^2}{2}  \Rr),
$$
this proves the lemma.
\end{proof}
\begin{lem}
\label{l:fourier2}	
For any integer $m$, there exists a constant $C$ such that if if the weak derivatives of order $m$ of $f$ are in $L^2(\R^d)$, then
$$
(2\pi)^{-d} \int \Ll(1+|\xi|^{2m}\Rr) \ \Ll|\hat{f}(\xi)\Rr|^2 \ \d \xi \le \|f\|_2^2 + \sum_{j = 1}^d \|\dr_{x_j^m} f\|_2^2.
$$
\end{lem}
\begin{proof}
One has
$$
(-i \xi_j)^m \hat{f} = \widehat{\dr_{x_j^m} f}
$$
in the sense of distributions. By Parseval's theorem (\cite[Theorem~IX.6]{rs-fourier}), since ${\dr_{x_j^m} f}$ is assumed to be in $L^2(\R^d)$, so is $\widehat{\dr_{x_j^m} f}$, and 
$\|\widehat{\dr_{x_j^m} f}\|_2 = (2\pi)^{d/2} \|\dr_{x_j^m} f\|_2$. Hence,
$$
\int |\xi_j|^{2m} |\hat{f}(\xi)|^2 \ \d \xi = (2\pi)^d \|\dr_{x_j^m} f\|_2^2,
$$
and as a consequence,
$$
\int |\xi|^{2m} |\hat{f}(\xi)|^2 \ \d \xi \le (2\pi)^d  \sum_{j = 1}^d \|\dr_{x_j^m} f\|_2^2.
$$
One also has $\|\hat{f}\|_2 = (2\pi)^{d/2} \|f\|_2$, so the lemma is proved.
\end{proof}
\begin{rem}
In fact, as the proof reveals, there is a converse to the lemma: if 
$$
\int (1+|\xi|^{2m}) \ |\hat{f}(\xi)|^2 \ \d \xi
$$
is finite, then all the weak derivatives of $f$ up to order $m$ are in $L^2(\R^d)$.
\end{rem}
\begin{proof}[Proof of Theorem~\ref{t:main}]
Let $t > 0$. We saw in the proof of Theorem~\ref{t:homog} that
\begin{eqnarray*}
\E[u_\eps(t,x)] & = & \E \EE^{\theta_{\lfloor \eps^{-1} x \rfloor}  \omega}_0 [f(\eps X_{\eps^{-2} t} + x_\eps)] \\
& = & \ov{\E}[f(\eps X_{\eps^{-2} t} + x_\eps)],
\end{eqnarray*}
where in the last line, we used the fact that the measure $\P$ is translation invariant, and we recall that we write $\ov{\E}$ for $\E\EEo_0$ and $x_\eps$ for $\eps \lfloor \eps^{-1} x \rfloor$. Note that
$$
\Ll|\ov{\E}[f(\eps X_{\eps^{-2} t} + x_\eps)] - \ov{\E}[f(\eps X_{\eps^{-2} t} + x)]\Rr| \le \sum_{j=1}^d \|\dr_{x_j} f \|_\infty \ \eps,
$$
which is the first term in the r.h.s.\ of \eqref{e:main} (a ``lattice effect''). We now focus on studying the difference
$$
\Ll| \ov{\E}[f(\eps X_{\eps^{-2} t} + x)]  - \EE_0[f(B_t + x)]\Rr|,
$$
where we recall that $\EE_0[f(B_t+x)] = \EE_x[f(B_t)] = \ov{u}(t,x)$. Possibly replacing $f$ by $f(\ \cdot\ + x)$, we may as well suppose that $x = 0$. Let $\sigma > 0$ be a small parameter, $\mcl{N}$ be a standard $d$-dimensional Gaussian random variable, independent of everything else, and write $f_t = f(\sqrt{t} \ \cdot)$. Since $f_t$ is bounded and continuous, we have
\begin{equation}
\label{limsigma}
\ov{\E}[f(\eps X_{\eps^{-2} t})] = \ov{\E}\Ll[f_t\Ll(\frac{\eps}{\sqrt{t}} X_{\eps^{-2} t}\Rr)\Rr] =  \lim_{\sigma \to 0} \ov{\E}\Ll[f_t\Ll(\frac{\eps}{\sqrt{t}} X_{\eps^{-2} t}+ \sigma \mcl{N}\Rr)\Rr].
\end{equation}
Similarly,
\begin{equation}
\label{limsigma2}
\EE_0[f(B_t)] = \EE_0[f(\sqrt{t} B_1)] = \EE_0[f_t(B_1)] =  \lim_{\sigma \to 0} \EE_0[f_t(B_1 + \sigma \mcl{N})],
\end{equation}
where we slightly abuse notation by using the same $\mcl{N}$ to denote a standard Gaussian (independent of everything else) under both the measures $\EE_0$ and $\ov{\E}$.
Let us write $\nu_{\eps}$ for the distribution of 
$$
\frac{\eps}{\sqrt{t}} X_{\eps^{-2} t}
$$
under the measure $\ov{\P}$, and $\nu_0$ for the distribution of $B_1$ under $\EE_0$. Note that
\begin{equation*}
\hat{\nu}_\eps(\xi) = \ov{\E}\Ll[ \exp\Ll( i |\xi| \ \frac{\eps \ \xi \cdot X_{\eps^{-2} t}}{ \sqrt{t} \ |\xi|} \Rr) \Rr].
\end{equation*}
The function $\R \to \R, x \mapsto e^{i |\xi| x}$ has first derivative bounded by $|\xi|$ and second derivative bounded by $|\xi|^2$.In view of \eqref{compfl1k}, we obtain from Theorem~\ref{t:rw} that
$$
\Ll|\hat{\nu}_\eps(\xi) - \hat{\nu}_0(\xi) \Rr| \le C |\xi| \ (|\xi| \vee 1) \ \Psi_{q,\delta}\Ll(\frac{\eps^2}{{t}}\Rr).
$$
Using Lemma~\ref{l:fourier}, we thus obtain that
\begin{equation*}
\begin{split}
& \Ll| \ov{\E}\Ll[f_t\Ll(\frac{\eps}{\sqrt{t}} X_{\eps^{-2} t}+ \sigma \mcl{N}\Rr)\Rr] - \EE_0[f_t(B_1 + \sigma \mcl{N})] \Rr| \\
& \qquad \le (2\pi)^{-d} \int_{\R^d} \exp\Ll(-\frac{\sigma^2 |\xi|^2}{2}\Rr) \hat{f}_t(\xi) \Ll|\hat{\nu}_\eps(\xi) - \hat{\nu}_0(\xi) \Rr| \ \d \xi \\
& \qquad \le C \ \Psi_{q,\delta}\Ll(\frac{\eps^2}{{t}}\Rr) \underbrace{\int \Ll|\hat{f}_t(\xi)\Rr| \ |\xi| \ (|\xi| \vee 1)   \ \d \xi},
\end{split}
\end{equation*}
where $C$ does not depend on $\sigma$.
Since $\hat{f}_t(\xi) = t^{-d/2} \hat{f}(\xi/\sqrt{t})$, we can perform a change of variables on the integral underbraced above, and bound it by 
$$
(t +\sqrt{t}) \int  \Ll|\hat{f}(\xi)\Rr| \  (|\xi|^2 + 1)   \ \d \xi.
$$
Let $m = \lfloor d/2 \rfloor +3$. By the Cauchy-Schwarz inequality, the integral above is bounded by
$$
\Ll(\int \frac{(|\xi|^2 + 1)^2}{{1+|\xi|^{2m}}} \ \d \xi \Rr)^{1/2}  \Ll(\int \Ll(1+|\xi|^{2m}\Rr) \ \Ll|\hat{f}(\xi)\Rr|^2 \ \d \xi \Rr)^{1/2}.
$$
Since $2m - 4 > d$, the first term of this product is finite, while Lemma~\ref{l:fourier2} gives us that the second term is bounded by
$$
(2\pi)^{d/2}\Ll(\|f\|_2^2 + \sum_{j = 1}^d \|\dr_{x_j^m} f\|_2^2\Rr)^{1/2} \le (2\pi)^{d/2} \Ll(\|f\|_2 + \sum_{j = 1}^d \|\dr_{x_j^m} f\|_2\Rr).
$$
We have thus proved that
\begin{multline*}
\Ll| \ov{\E}\Ll[f_t\Ll(\frac{\eps}{\sqrt{t}} X_{\eps^{-2} t}+ \sigma \mcl{N}\Rr)\Rr] - \EE_0[f_t(B_1 + \sigma \mcl{N})] \Rr| \\
\le C \  (t+\sqrt{t})  \Ll(\|f\|_2 +   \sum_{j=1}^d \|\dr_{x_j^m} f\|_2 \Rr) \ \Psi_{q,\delta}\Ll(\frac{\eps^2}{{t}}\Rr).
\end{multline*}
Taking the limit $\sigma \to 0$ and recalling \eqref{limsigma} and \eqref{limsigma2}, we obtain the announced result.
\end{proof}
%
%
%
%
%
%
%
%
\section{Heat kernel estimates}
\label{s:hk}
\setcounter{equation}{0}
The heat kernel $p_t^\o(x,y)$ is defined so that $(t,y) \mapsto p^\o_t(x,y)$ is the unique bounded solution to \eqref{parabeq} with initial condition $f = \1_x$. The heat kernel is symmetric: $p_t^\o(x,y) = p_t^\o(y,x)$, and by translation invariance of the random coefficients, $\E [p_t^\o(x,y)] = \E [p_t^\o(0,y-x)]$.

The aim of this section is to prove Theorem~\ref{t:main2}. In order to do so, we will need a regularity result on the averaged heat kernel. For $f : \Z^d \to \R$ and $1 \le i \le d$, we write
$$
\nabla_i f(x) = f(x+\mathbf{e}_i) - f(x),
$$
where $(\mathbf{e}_i)_{1 \le i \le d}$ is the canonical basis of $\R^d$. The following result was proved in \cite[Theorem~1.4]{connadny}, and then elegantly rederived in \cite[(1.4)]{deldeu}.
\begin{thm}[\cite{connadny,deldeu}]
\label{t:diffqt}
Let 
\begin{equation}
\label{defqt}
q_t(x) = \E\Ll[ p^\o_t(0,x) \Rr]. 
\end{equation}
There exist $C, c_1 > 0$ such that for any $t > 0$ and any $x \in \Z^d$, one has
$$
\Ll| \nabla_i q_t(x) \Rr| \le \frac{C}{t^{(d+1)/2}} \exp\Ll( -c_1\Ll( \frac{|x|^2}{t} \wedge |x| \Rr) \Rr).
$$
\end{thm}
We also recall the following upper bound on the heat kernel, taken from \cite[Proposition~3.4]{del} (see also \cite[Section~3]{cks} for earlier results in this context).
\begin{thm}[\cite{del}]
\label{t:del}
(1) There exist constants $C, \ov{c}$ such that for any $t \ge 0$ and any $x \in \Z^d$,
$$
p_t^\omega(0,x) \le \frac{C}{1 \vee t^{d/2}} \exp\Ll(-D_{\ov{c}t}(x)\Rr),
$$
where 
$$
D_{t}(x) = |x|\arsinh\Ll( \frac{|x|}{t} \Rr) + t \Ll(\sqrt{1+\frac{|x|^2}{t^2}} - 1\Rr).
$$
(2) In particular, there exists $c_2 > 0$ such that for any $x\in \Z^d$,
$$
p_t^\omega(0,x) \le \frac{C}{1 \vee t^{d/2}} \exp\Ll(-c_2\Ll( \frac{|x|^2}{t} \wedge |x| \Rr) \Rr).
$$
\end{thm}
\begin{proof}[Proof of Theorem~\ref{t:main2}]
We decompose the proof into three steps.

\medskip

\noindent \emph{Step 1}. Possibly lowering the value of $c_2 > 0$, we have that for any $x \in \R^d$,
\begin{equation}
\label{homogeneousbound}
\ov{p}_1(0,x) \le C \exp\Ll(-c_2 |x|^2\Rr),
\end{equation}
\begin{equation}
\label{homogeneousgrad}
\frac{\dr \ov{p}_1(0,\cdot)}{\dr x_i}(x) \le C \exp\Ll(-c_2 |x|^2\Rr) \qquad (1 \le i \le d).
\end{equation}
Equation \eqref{homogeneousbound} and part (2) of Theorem~\ref{t:del} thus ensure that (possibly enlarging $C$),
\begin{equation}
\label{time10}
\Ll| \eps^{-d} \ q_{\eps^{-2}}(\lfloor \eps^{-1} x \rfloor) - \ov{p}_1(0,x) \Rr| \le C \exp\Ll(-c_2 (|x|^2 \wedge |\eps^{-1} x|)  \Rr).
\end{equation}
Moreover, Theorem~\ref{t:diffqt} remains true if we lower the value of the constant $c_1 > 0$ in such a way that $c_2 \ge c_1/2\sqrt{d}$. 

\medskip

\noindent \emph{Step 2}. We now show that there exist $c > 0$ (independent of $\delta$), $\eps_\delta > 0$ and $C_\delta$ such that, for any $\eps \le \eps_\delta$ and any $x \in \R^d$, one has
\begin{equation}
\label{time1}
\Ll| \eps^{-d} \ q_{\eps^{-2}}(\lfloor \eps^{-1} x \rfloor) - \ov{p}_1(0,x) \Rr| \le C_\delta \Ll(\Psi_{q,\delta}(\eps^2)\Rr)^{1/(d+3)}\exp\Ll(-c (|x|^2 \wedge |\eps^{-1} x|)  \Rr).
\end{equation}
Let $f$ be a positive smooth function on $\R^d$ with support in $[-1,1]^d$ and such that $\int f = 1$. We define, for any $r > 0$, the function $f_r : x \mapsto r^{-d} f(r^{-1} x)$.

Let $u^{(\eps)}$ be the bounded solution of \eqref{parabeq} with initial condition $f_r(\eps \ \cdot)$ (we keep the dependence of $u^{(\eps)}$ in $r$ implicit in the notation). By linearity, we have
$$
u^{(\eps)}(t,x) = \sum_{z \in \Z^d} f_r(\eps z) \ p_t^\o(z,x).
$$
Letting $u_\eps(t,x) = u^{(\eps)}(\eps^{-2} t, \lfloor \eps^{-1} x \rfloor)$, we obtain
\begin{equation}
\label{repueps}
u_\eps(t,x) = \sum_{z \in \Z^d} f_r(\eps z) \ p_{\eps^{-2}t}^\o(z,\lfloor \eps^{-1} x \rfloor).
\end{equation}
Let $\ov{u}$ be the bounded solution of \eqref{cparabeq} with initial condition $f_r$. Observing the proof of Theorem~\ref{t:main}, we get that for any $\delta > 0$, there exists $C$ such that
\begin{multline}
\label{e:frommain1}
\Ll|\E[u_\eps(1,x)] - \ov{u}(1,x) \Rr| \\
\le \sum_{j=1}^d \|\dr_{x_j} f_r \|_\infty \ \eps +  C \   \Psi_{q,\delta}({\eps^2}) \  \int  \Ll|\hat{f_r}(\xi)\Rr| \  (|\xi|^2 + 1)   \ \d \xi.
\end{multline}
Scaling relations ensures that $\|\dr_{x_j} f_r \|_\infty$ is bounded, up to a constant, by $r^{-(d+1)}$, while $\hat{f_r}(\xi) = \hat{f}(r \xi)$. As a consequence, 
$$
\int  \Ll|\hat{f_r}\Rr|  = r^{-d} \int \Ll|\hat{f}\Rr|,
$$
$$
\int  \Ll|\hat{f_r}(\xi)\Rr| \  |\xi|^2   \ \d \xi = r^{-(d+2)} \int \Ll|\hat{f}(\xi)\Rr| \ |\xi|^2 \ \d \xi,
$$
and the integrals on the r.h.s.\ are finite since $f$ is smooth (see Lemma~\ref{l:fourier2}). To sum up, for some constant $C$ and any $r \le 1$, we have
\begin{equation}
\label{ecartuepsovu}
\Ll|\E[u_\eps(1,x)] - \ov{u}(1,x) \Rr| \le C \Ll(
\eps \ r^{-(d+1)} + \Psi_{q,\delta}\Ll({\eps^2}\Rr) r^{-(d+2)}
\Rr).
\end{equation}
The solution $\ov{u}$ can be represented in terms of the heat kernel as
$$
\ov{u}(1,x) = \int f_r(z) \ov{p}_1(z,x) \ \d z = \ov{p}_1(0,x) + \int f_r(z) (\ov{p}_1(z,x) - \ov{p}_1(0,x)) \ \d z,
$$
where we used the fact that $\int f_r = 1$. For $z \in \R^d$ such that $\|z\|_\infty \le r\le 1$ and up to a constant, $|\ov{p}_1(z,x) - \ov{p}_1(0,x)|$ is bounded by $r e^{-c_2 |x|^2}$ by \eqref{homogeneousgrad}. Since $f_r$ has support in $[-r,r]^d$, we arrive at
\begin{equation}
\label{step2}
\Ll| \ov{u}(1,x) - \ov{p}_1(0,x) \Rr| \le C \ {r} \exp\Ll(-c_2 |x|^2\Rr).
\end{equation}
On the other hand, if $z \in \Z^d$ is such that $\|z\|_\infty \le \eps^{-1} r$, then
$$
\Ll| \E[p_{\eps^{-2}}^\o(z,\lfloor \eps^{-1} x \rfloor)] - q_{\eps^{-2}}(\lfloor \eps^{-1} x \rfloor) \Rr| \le d \eps^{-1} r \sup_{\substack{\|z\|_\infty \le \eps^{-1} r \\ 1 \le i \le d}} |\nabla_i q_{\eps^{-2}}(\lfloor \eps^{-1} x \rfloor - z)|
$$
We now argue that there exists $c_3 > 0$ (independent of $\delta$) such that, uniformly over $r \le 1$ and $x \in \R^d$, one has
\begin{equation}
\label{controlgrad}
\sup_{\substack{\|z\|_\infty \le \eps^{-1} r \\ 1 \le i \le d}} |\nabla_i q_{\eps^{-2}}(\lfloor \eps^{-1} x \rfloor - z)| \le \frac{C}{\eps^{d+1}} \exp\Ll[-c_3\Ll( |x|^2 \wedge |\eps^{-1} x| \Rr) \Rr].
\end{equation}
Theorem~\ref{t:diffqt} tells us indeed that the l.h.s.\ of \eqref{controlgrad} is smaller than
$$
\frac{C}{\eps^{d+1}} \exp\Ll[ -c_1\inf_{\substack{\|z\|_\infty \le \eps^{-1} r \\ 1 \le i \le d}} \Ll(\frac{|\lfloor \eps^{-1} x \rfloor - z|^2}{\eps^{-2}} \wedge |\lfloor \eps^{-1} x \rfloor - z | \Rr) \Rr].
$$
For any $r \le 1$ and $\|x\|_\infty \ge 2$, the infimum above is larger than 
$$
\frac{|x|^2 \wedge |\eps^{-1} x|}{2\sqrt{d}},
$$
so \eqref{controlgrad} holds in this case, with $c_3 = c_1/2\sqrt{d}$. To control smaller values of $\|x\|_\infty$, it suffices to enlarge the constant $C$ in \eqref{controlgrad}.
To sum up, we have shown that
$$
\Ll| \E[p_{\eps^{-2}}^\o(z,\lfloor \eps^{-1} x \rfloor)] - q_{\eps^{-2}}(\lfloor \eps^{-1} x \rfloor) \Rr|\le C \  \eps^d \ {r} \exp\Ll[ -c_3\Ll( |x|^2 \wedge |\eps^{-1} x| \Rr) \Rr].
$$

In the sum on the r.h.s.\ of \eqref{repueps}, only $C (\eps^{-1} r)^d$ terms are non-zero, and $\|f\|_\infty \le r^{-d}$, so
$$
\Ll| \E[u_\eps(1,x)] -  \sum_{z \in \Z^d} f_r(\eps z) \ q_{\eps^{-2}}(\lfloor \eps^{-1} x \rfloor) \Rr| \le C \ r\exp\Ll[ -c_3\Ll( |x|^2 \wedge |\eps^{-1} x| \Rr) \Rr].
$$
Observe also that 
$$
\eps^d \sum_{z \in \Z^d} f_r(\eps z) = \Ll(\frac{\eps}{r}\Rr)^d\sum_{z \in \Z^d} f_r\Ll(\frac{\eps}{r} z \Rr).
$$
This is a Riemann approximation of $\int f = 1$, hence
$$
\Ll| \eps^d \sum_{z \in \Z^d} f_r(\eps z) - 1 \Rr| \le C \ \frac{\eps}{r},
$$
and we are thus led to
\begin{equation}
\label{step3}
\Ll| \E[u_\eps(1,x)] - \eps^{-d} \  q_{\eps^{-2}}(\lfloor \eps^{-1} x \rfloor) \Rr| \le C \Ll({r}\exp\Ll[ -c_3\Ll( |x|^2 \wedge |\eps^{-1} x| \Rr) \Rr] + \frac{\eps}{r}\Rr).
\end{equation}
Combining \eqref{ecartuepsovu}, \eqref{step2}, \eqref{step3} and the fact that $c_2 \ge c_3 = c_1/2\sqrt{d}$, we obtain that up to a constant, 
$$
\Ll|\eps^{-d}\  q_{\eps^{-2}}(\lfloor \eps^{-1} x \rfloor) - \ov{p}_1(0,x) \Rr|
$$
is bounded by
$$
\frac{\eps}{r^{d+1}} + \frac{\Psi_{q,\delta}\Ll({\eps^2}\Rr)}{r^{d+2}} + r\exp\Ll[ -c_3\Ll( |x|^2 \wedge |\eps^{-1} x| \Rr) \Rr] + \frac{\eps}{r},
$$
uniformly over $r \le 1$. Since for $\eps$ small enough, one has $\eps \le \Psi_{q,\delta}\Ll({\eps^2}\Rr)$, the above is bounded, up to a constant, by
\begin{equation}
\label{bound1}
\frac{\Psi_{q,\delta}\Ll({\eps^2}\Rr)}{r^{d+2}} + r \exp\Ll[ -c_3\Ll( |x|^2 \wedge |\eps^{-1} x| \Rr) \Rr],
\end{equation}
uniformly over $r \le 1$. Choosing 
$$
r^{d+3} = \Psi_{q,\delta}(\eps^2) \exp\Ll[ c_3\Ll( |x| \wedge |\eps^{-1} x| \wedge M_\eps \Rr) \Rr], 
$$
where
$$
M_\eps = -\frac{\log(\Psi_{q,\delta}(\eps^2))}{c_3}
$$
is here to ensure that $r \le 1$, 
we obtain that the expression in \eqref{bound1} is smaller than
$$
\Ll(\Psi_{q,\delta}(\eps^2)\Rr)^{1/(d+3)} \exp\Ll[ -c_3\Ll(1-\frac{1}{d+3}\Rr)\Ll( |x| \wedge |\eps^{-1} x| \wedge M_\eps \Rr) \Rr].
$$
This proves \eqref{time1} when $|x| \wedge |\eps^{-1} x| \le M_\eps$. Otherwise, we use the bound \eqref{time10}, together with the fact that $c_2 \ge c_3$, to get
\begin{equation*}
\begin{split}
& \Ll| \eps^{-d} \ q_{\eps^{-2}}(\lfloor \eps^{-1} x \rfloor) - \ov{p}_1(0,x) \Rr| \\ 
& \qquad  \le  C \exp\Ll(-c_3 (|x|^2 \wedge |\eps^{-1} x|)  \Rr) \\
& \qquad  \le  C \exp\Ll(-c_3 \Ll(1-\frac{1}{d+3}\Rr)(|x|^2 \wedge |\eps^{-1} x|) - \frac{c_3}{d+3} M_\eps\Rr) \\
& \qquad  \le  C \Ll(\Psi_{q,\delta}(\eps^2)\Rr)^{1/(d+3)} \exp\Ll(-c_3 \Ll(1-\frac{1}{d+3}\Rr)(|x|^2 \wedge |\eps^{-1} x|)\Rr).
\end{split}
\end{equation*}
Hence, \eqref{time1} holds also in this case, and we can always choose $c = c_3(1-1/(d+3))$.

\medskip

\noindent \emph{Step 3}. We now extend the result to any time $t > 0$. The heat kernel of the continuous operator satisfies the scaling relation
$$
\ov{p}_t(0,x) = t^{-d/2} \ \ov{p}_1(0,x/\sqrt{t}),
$$
while we can write
$$
\eps^{-d} \ q_{\eps^{-2}t}(\lfloor \eps^{-1} x \rfloor) = t^{-d/2} \  (\eps/ \sqrt{t})^{-d} \  q_{(\eps/ \sqrt{t})^{-2}}(\lfloor (\eps/ \sqrt{t})^{-1} \ (x/\sqrt{t}) \rfloor).
$$
For $\eps_\delta$ and $C_\delta$ given by step 2, as soon as $\eps/ \sqrt{t} \le \eps_\delta$, one thus has
\begin{multline*}
\Ll| \eps^{-d} \ q_{\eps^{-2}t}(\lfloor \eps^{-1} x \rfloor) - \ov{p}_t(0,x) \Rr| \\
\le \frac{C_\delta}{t^{d/2}} \ \Ll(\Psi_{q,\delta}\Ll(\frac{\eps^2}{t}\Rr)\Rr)^{1/(d+3)}\exp\Ll[ -c\Ll( \frac{|x|^2}{t} \wedge |\eps^{-1} x| \Rr) \Rr],
\end{multline*}
which is the claim of the theorem.
\end{proof}
%
%
%
%
%
%
%
%
\section{Homogenization of elliptic equations}
\label{s:ellipt}
\setcounter{equation}{0}

In this last section, we state and prove the counterparts of Theorems~\ref{t:main} and \ref{t:main2} for the homogenization of elliptic equations.  For $f : \R^d \to \R$ bounded continuous, we consider the unique bounded solution of
\begin{equation*}
\label{elleqeps}
\tag{DEE$^\omega_\eps$}
(\eps^2-L^\omega) v^{(\eps)} = \eps^2 f(\eps \ \cdot) \quad \text{on } \Z^d.
\end{equation*}
Using integration by parts, one can check that
\begin{equation}
\label{integrparab0}
v^{(\eps)}(x) = \int_0^{+\infty} e^{-t} \ u^{(\eps)}(\eps^{-2} t,x) \ \d t,
\end{equation}
where $u^{(\eps)}$ is solution of \eqref{parabeqeps}. For $x \in \R^d$, we let $v_\eps(x) = v^{(\eps)}(\lfloor \eps^{-1} x \rfloor)$, so that
\begin{equation}
\label{integrparab}
v_\eps(x) = \int_0^{+\infty} e^{-t} \ u_\eps(t,x) \ \d t.
\end{equation}
The function $v_\eps$ converges pointwise, as $\eps$ tends to $0$, to $\ov{v}$ the bounded solution of 
\begin{equation*}
\label{celleq}
\tag{CEE}
\Ll(1-\frac{1}{2} \nabla \cdot \ov{A} \nabla\Rr) \ov{v} = f \quad \text{on } \R^d,
\end{equation*}
and one has
\begin{equation}
\label{integrcontinu}
\ov{v}(x) = \int_0^{+\infty} e^{-t} \ \ov{u}(t,x) \ \d t,
\end{equation}
where $\ov{u}$ is the solution of \eqref{cparabeq}. Equipped with the representations \eqref{integrparab}-\eqref{integrcontinu}, it is straightforward to derive the following result from Theorem~\ref{t:main}.
\begin{thm}
\label{t:mainel}
Let $m = \lfloor d/2 \rfloor + 3$ and $\delta > 0$. There exist constants $C_\delta$ (which may depend on the dimension) and $q$ such that, if the weak derivatives of order $m$ of $f$ are in $L^2(\R^d)$, then for any $\eps > 0$ and $x \in \R^d$, one has
\begin{equation*}
\label{e:mainel}
\Ll|\E[v_\eps(x)] - \ov{v}(x) \Rr| \\
\le \sum_{j=1}^d \|\dr_{x_j} f \|_\infty \ \eps +  C_\delta   \Ll(\|f\|_2 +   \sum_{j=1}^d \|\dr_{x_j^m} f\|_2 \Rr) \ \Psi_{q,\delta}\Ll({\eps^2}\Rr).
\end{equation*}
\end{thm}
\begin{rem}
Note that on the other hand, it does not look so simple to deduce Theorem~\ref{t:main} from Theorem~\ref{t:mainel}. A possibility for doing so may be to try to devise a quantitative version of \cite[Theorem IX.2.16]{kato}.
\end{rem}
One can also consider the Green function $G_\eps^\omega(x,y)$, the unique bounded function such that
$$
(\eps^2-L^\omega) G_\eps^\omega(x,\cdot) = \1_x.
$$
Letting $\ov{G}(x,y)$ be the Green function associated to equation \eqref{celleq}, we can write the counterpart of Theorem~\ref{t:main2}.
\begin{thm}
\label{t:mainelhk}
Let $d \ge 2$ and $\delta > 0$. There exist constants $c > 0$ (independent of $\delta$), $q$, $C_\delta$ such that for any $\eps > 0$ and any $x \in \eps \Z^d \setminus \{0\}$, one has
\begin{multline}
\label{e:mainelhk}
\Ll| \eps^{2-d} \ \E\Ll[ G_\eps^\omega(0,\eps^{-1} x)\Rr] - \ov{G}(0,x) \Rr| \\
\le \frac{C_\delta}{|x|^{d-2}} \Ll[ \Ll(\Psi_{q,\delta}\Ll(\frac{\eps^2}{|x|^2}\Rr)\Rr)^{1/(d+3)} e^{-c|x|} + e^{-c|\eps^{-1}x|} \Rr].
\end{multline}
When $d = 1$, there exist $C, c> 0$ such that for any $\eps > 0$ and any $x \in \eps \Z$, one has
$$
\Ll| \eps \ \E\Ll[ G_\eps^\omega(0,\eps^{-1} x)\Rr] - \ov{G}(0,x) \Rr| \le C \Ll[ \eps^{1/8}  e^{-c|x|} +  e^{-c|\eps^{-1}x|} \Rr].
$$
\end{thm}
\begin{rem}
The orders of magnitude, as $\eps$ tends to $0$, of the r.h.s.\ of \eqref{e:mainel} and \eqref{e:mainelhk}, are given respectively by \eqref{boundrough1} and \eqref{boundrough2}.
\end{rem}
\begin{proof}
Our starting point is the fact that 
$$
G_\eps^\omega(x,y) = \eps^{-2} \int_0^{+\infty} e^{-t} \ p^\omega_{\eps^{-2}t}(x,y) \ \d t,
$$
while
$$
\ov{G}(x,y) = \int_0^{+\infty} e^{-t} \ \ov{p}_t(x,y) \ \d t
$$
Recall first that Theorem~\ref{t:main2} ensures that there exist $c >0, C_\delta,\eps_\delta > 0$ such that whenever $t\ge (\eps/\eps_\delta)^2$, one has
\begin{multline}
\label{boundfromthm2}
\Ll|\eps^{-d} \E[p_{\eps^{-2} t}^\omega(0,\lfloor \eps^{-1} x \rfloor)] - \ov{p}_t(0,x)\Rr| \\
\le \frac{C_\delta}{t^{d/2}} \Psi_{q,\delta}^{1/(d+3)}\Ll(\frac{\eps^2}{t}\Rr) \exp\Ll[ -c \Ll(\frac{|x|^2}{t} \wedge |\eps^{-1} x| \Rr) \Rr].
\end{multline}
The difference of interest
$$
\Ll| \eps^{2-d} \ \E\Ll[ G_\eps^\omega(0,\lfloor \eps^{-1} x \rfloor)\Rr] - \ov{G}(0,x) \Rr|
$$
is bounded by
\begin{equation}
\label{integraltobound}
\int_0^{+\infty} e^{-t} \Ll|\eps^{-d} \ \E[p_{\eps^{-2} t}^\omega(0,\lfloor \eps^{-1} x \rfloor)] - \ov{p}_t(0,x)  \Rr| \ \d t.
\end{equation}
Let $\eta = (\eps/\eps_\delta)^2 \vee (\eps |x|)$. If $t \ge \eta$, then the integrand above is bounded, up to a constant, by
$$
\frac{e^{-t}}{t^{d/2}} \Psi_{q,\delta}^{1/(d+3)}\Ll(\frac{\eps^2}{t}\Rr) \exp\Ll[ -c \frac{|x|^2}{t}\Rr].
$$
In order to control the integral in \eqref{integraltobound}, it thus suffices to bound the following three quantities:
\begin{equation}
\label{int1}	
\int_{0}^{+\infty} \frac{e^{-t}}{t^{d/2}}  \Psi_{q,\delta}^{1/(d+3)}\Ll(\frac{\eps^2}{t}\Rr) \exp\Ll[ -c \frac{|x|^2}{t} \Rr] \ \d t,
\end{equation}
\begin{equation}
\label{int2}
\int_0^\eta \eps^{-d} \E[p_{\eps^{-2} t}^\omega(0,\lfloor \eps^{-1} x \rfloor)] \ \d t,
\end{equation}
\begin{equation}
\label{int3}
\int_0^\eta \ov{p}_t(0,x) \ \d t.
\end{equation}
We start with the integral in \eqref{int1}, which is the only non-negligible one. To begin with, note that for any $\gamma$, a change of variables gives us the identity
\begin{equation}
\label{changevar}
\int_0^{+\infty}\frac{e^{-t}}{t^{\gamma}} e^{-c |x|^2/t} \ \d t = |x|^{2-2\gamma} \int_0^{+\infty} \frac{e^{-s|x|^2}}{s^{\gamma}} e^{-c/s} \ \d s,
\end{equation}
and moreover, provided $\gamma > 1$,
\begin{eqnarray}
\label{boundgamma1}
\int_0^{+\infty} \frac{e^{-s|x|^2}}{s^{\gamma}} e^{-c/s} \ \d s & \le & e^{-c|x|/2} \int_0^{1/|x|} \frac{e^{-c/2s}}{s^\gamma} \ \d s + e^{-|x|} \int_{1/|x|}^{+\infty} \frac{e^{-c/s}}{s^\gamma} \ \d s \notag \\
& \le & C e^{-c|x|/2},
\end{eqnarray}
for some large enough $C$ (and $c \le 2$). We have thus shown that, for $\gamma > 1$,
\begin{equation}
\label{estim1}
\int_0^{+\infty}\frac{e^{-t}}{t^{\gamma}} e^{-c |x|^2/t} \ \d t \le C |x|^{2 - 2\gamma} e^{-c|x|/2}.
\end{equation}
When $d \ge 3$, we have $\Psi_{q,\delta}(u) = u^{1/2-\delta}$, so that the integral in \eqref{int1} is bounded, up to a constant, by
\begin{equation}
\label{boundonint1}
|x|^{2-d} \ \Psi_{q,\delta}^{1/(d+3)}\Ll( \frac{\eps^2}{|x|^2} \Rr) e^{-c|x|/2}.
\end{equation}
When $d = 2$, the argument requires some minor modifications, due to presence of a logarithmic factor in $\Psi_{q,\delta}$. One should consider instead integrals of the form
$$
\int_0^{+\infty}\frac{e^{-t}}{t^{\gamma}} \log_+^{q'}\Ll(t/\eps^2\Rr) e^{-c |x|^2/t} \ \d t = |x|^{2-2\gamma} \int_0^{+\infty} \frac{e^{-s|x|^2}}{s^{\gamma}} \log_+^{q'}\Ll(s|x|^2/\eps^2\Rr) e^{-c/s} \ \d s,
$$
for some $q' \ge 0$ and $\gamma > 1$ (in fact, $\gamma = 1+1/20$).
This last integral is bounded by 
\begin{equation*}
\int_0^{\eps^2/|x|^2}\frac{e^{-s|x|^2}}{s^{\gamma}}  e^{-c/s} \ \d s + \int_{\eps^2/|x|^2}^{+\infty} \frac{e^{-s|x|^2}}{s^{\gamma}} \log^{q'}\Ll(s|x|^2/\eps^2\Rr) e^{-c/s} \ \d s
\end{equation*}
For the first integral, \eqref{boundgamma1} gives us an upper bound. Inequality \eqref{boundgamma1} also enables us to bound the second integral, using the fact that 
$$
\log^{q'}\Ll(s|x|^2/\eps^2\Rr) \le 2^{q'} \Ll( \log^{q'}\Ll(|x|^2/\eps^2\Rr) + \log^{q'}\Ll(s\Rr) \Rr).
$$
These observations thus guarantee that \eqref{int1} is also bounded by \eqref{boundonint1} when $d = 2$. 

We now turn to the evaluation of the integral in \eqref{int2}. 
Since, for $z \ge 0$, one has $\arsinh(z) = \log(z+\sqrt{1+z^2}) \ge \log(1+z)$, and using part (1) of Theorem~\ref{t:del}, one can bound the integral in \eqref{int2} (up to a constant) by
$$
\int_0^\eta \eps^{-d} \exp\Ll( - |\eps^{-1} x| \log\Ll(1+\frac{|\eps^{-1} x|}{\ov{c} \eps^{-2} t}\Rr) \Rr) \ \d t.
$$
A change of variables shows that this is equal to 
\begin{equation}
\label{int22}
\frac{\eps |x|}{\ov{c}} \eps^{-d}  \int_0^{\eta'} \exp\Ll( - |\eps^{-1} x| \log\Ll(1+1/s\Rr) \Rr) \ \d s,
\end{equation}
where 
$$
\eta' = \frac{\ov{c} \eta}{\eps |x|} = \frac{\ov{c} \eps_\delta^{-2}}{|\eps^{-1} x|} \vee \ov{c}.
$$
Since we consider only $x \in \eps \Z^d \setminus \{0\}$, the parameter $\eta'$ is uniformly bounded, independently of the value of $x$ and $\eps$. The integral in \eqref{int22} is thus bounded (up to a constant) by
\begin{multline*}
\eps^{1-d} |x|  (1+\eta^{-1})^{-|\eps^{-1} x|} =  |x|^{2-d} |\eps^{-1} x|^{d-1} (1+\eta^{-1})^{-|\eps^{-1} x|} \\
 \le C |x|^{2-d} \exp\Ll(-c|\eps^{-1} x|\Rr).
\end{multline*}
This finishes the analysis of the integral in \eqref{int2}, and there remains only to consider the integral in \eqref{int3}. This integral is bounded by a constant times
$$
\int_0^{\eta} t^{-d/2} e^{-c |x|^2/t} \ \d t
$$
for some small enough $c> 0$. A change of variables enables one to rewrite this integral as
\begin{equation}
\label{int32}
|x|^{2-d} \int_0^{\eta |x|^{-2} } u^{-d/2} e^{-c/u} \ \d u \\
\le |x|^{2-d}\exp\Ll( -\frac{c}{2\eta|x|^{-2}} \Rr) \int_0^{\eta |x|^{-2}} u^{-d/2} e^{-c/2u} \ \d u.
\end{equation}
Moreover,
$$
\eta |x|^{-2} = \frac{\eps_\delta^{-2}}{|\eps^{-1} x|^2} \vee \frac{1}{|\eps^{-1} x|} \le \frac{C'}{|\eps^{-1} x|} \le C'
$$
for some large enough $C'$, uniformly over $\eps > 0$ and $x \in \eps \Z^d \setminus \{0\}$. The r.h.s.\ of \eqref{int32} is thus bounded by
$$
|x|^{2-d} \exp\Ll( -\frac{|\eps^{-1} x|}{C'} \Rr) \int_0^{C'} u^{-d/2} e^{-c/2u} \ \d u.
$$
We thus obtained the required bound on \eqref{int3}, and this finishes the proof of Theorem~\ref{t:mainelhk} for $d \ge 2$.

\medskip 

For the one-dimensional case, the analysis must be slightly adapted. We need to bound the integrals appearing in \eqref{int1}, \eqref{int2} and \eqref{int3}. The analysis of the integrals in \eqref{int2} and \eqref{int3} can be kept without change, except that only the case $x \in \eps \Z \setminus \{0\}$ was considered above, while here we want to consider also $x = 0$. But this is a very easy case, since the upper bound $t^{-1/2}$ on the heat kernels is integrable close to $0$. As for the integral in \eqref{int1}, it is equal to
$$
\eps^{1/8} \int_0^{+\infty} \frac{e^{-t}}{t^\gamma} \ e^{-c|x|^2/t} \ \d t,
$$
where $\gamma = 1/2+1/16 < 1$. The integral above is uniformly bounded over $x$ such that $|x| \le 1$. Otherwise, as noted in \eqref{changevar}, we have
$$
\int_0^{+\infty} \frac{e^{-t}}{t^\gamma} \ e^{-c|x|^2/t} \ \d t = 
|x|^{2-2\gamma} \int_0^{+\infty} \frac{e^{-s|x|^2}}{s^{\gamma}} e^{-c/s} \ \d s,
$$
and we can bound the last integral by
$$
e^{-c|x|} \int_0^{1/|x|} \frac{e^{-s}}{s^\gamma} \ \d s + e^{-|x|/2} \int_{1/|x|}^{+\infty} \frac{e^{-s/2}}{s^\gamma} \ \d s,
$$
where in the second part, we used the fact that for $|x| \ge 1$ and $s \ge |x|^{-1}$, we have $s|x|^2 \ge |x|/2 + s/2$.
We have thus shown that the integral in \eqref{int1} is bounded, up to a constant, by
$$
\eps^{1/8} \Ll(|x|^{2-2\gamma} + 1\Rr)e^{-c|x|},
$$
uniformly over $x \in \R$, and this finishes the proof for $d = 1$.
\end{proof}


\begin{thebibliography}{99}

\bibitem[ABDH10]{andr} S.\ Andres, M.T.\ Barlow, J.-D.\ Deuschel, B.M.\ Hambly.
Invariance principle for the random conductance model. Preprint (2010).

\bibitem[BH09]{barham} M.T.\ Barlow, B.M.\ Hambly. Parabolic Harnack inequality and local limit theorem for percolation clusters. \emph{Electron. J. Probab.} \textbf{14} (1), 1–27 (2009). 

\bibitem[BLP]{blp} A.\ Bensoussan, J.-L.\ Lions, G.\ Papanicolaou. \emph{Asymptotic analysis for periodic structures}. Studies in mathematics and its applications \textbf{5}, North-Holland publishing (1978).

\bibitem[BBHK08]{bbhk} N.\ Berger, M.\ Biskup, C.E.\ Hoffman, G.\ Kozma. Anomalous heat-kernel decay for random walk among bounded random conductances. \emph{Ann. Inst. Henri Poincaré Probab. Stat.} \textbf{44} (2), 374–392 (2008).

\bibitem[BB10]{bbou} M.\ Biskup, O.\ Boukhadra. Subdiffusive heat-kernel decay in four-dimensional i.i.d.\ random conductance models. Preprint, 	arXiv:1010.5542v2 (2010).

\bibitem[Bo82]{bo} E.\ Bolthausen. Exact convergence rates in some martingale central limit theorems. \emph{Ann. Probab.} \textbf{10} (3), 672–688 (1982). 

\bibitem[CS10]{cafsou} L.A.\ Caffarelli, P.E.\ Souganidis. Rates of convergence for the homogenization of fully nonlinear uniformly elliptic pde in random media. \emph{Invent. Math.} \textbf{180} (2), 301–360 (2010).

\bibitem[CKS87]{cks} E.A.~Carlen, S.~Kusuoka, D.W.~Stroock. Upper bounds for symmetric Markov transition functions. \emph{Ann. Inst. Henri Poincaré Probab. Stat.} \textbf{23} (S2), 245-287 (1987).

\bibitem[CN00a]{connadny} J.\ Conlon, A.\ Naddaf. Green's functions for elliptic and parabolic equations with random coefficients. \emph{New York J. Math.} \textbf{6}, 153–225 (2000).

\bibitem[CN00b]{connadejp} J.\ Conlon, A.\ Naddaf. On homogenization of elliptic equations with random coefficients. \emph{Electron. J. Probab.} \textbf{5} (9), 58 pp (2000).

\bibitem[CS11]{conspe} J.\ Conlon, T.\ Spencer. Strong convergence to the homogenized limit of elliptic equations with random coefficients. Preprint, 	arXiv:1101.4914v1 (2011).

\bibitem[DFGW89]{masi} A.~De Masi, P.A.~Ferrari, S.~Goldstein, W.D.~Wick. An invariance principle for reversible Markov processes. Applications to random motions in random environments. \emph{J. Statist. Phys.} \textbf{55} (3-4), 787-855 (1989).

\bibitem[De99]{del} T.\ Delmotte. Parabolic Harnack inequality and estimates of Markov chains on graphs. \emph{Rev. Mat. Iberoamericana} \textbf{15} (1), 181–232 (1999).

\bibitem[DD05]{deldeu} T.\ Delmotte, J.-D.\ Deuschel. On estimating the derivatives of symmetric diffusions in stationary random environment, with applications to $\nabla \phi$ interface model. \emph{Probab. Theory Related Fields} \textbf{133} (3), 358–390 (2005). 

\bibitem[Dv72]{dvo} A.\ Dvoretzky. Asymptotic normality for sums of dependent random variables. \emph{Proceedings of the sixth Berkeley symposium on mathematical statistics and probability 1970/1971}, vol.~II: Probability theory, 513–535 (1972).

\bibitem[FM06]{fonmat} L.R.G.\ Fontes, P.\ Mathieu. On symmetric random walks with random conductances on $\Z^d$. \emph{Probab. Theory Related Fields} \textbf{134} (4), 565–602 (2006).

\bibitem[GO11]{glotto} A.\ Gloria, F.\ Otto. An optimal variance estimate in stochastic homogenization of discrete elliptic equations. \emph{Ann. Probab.} \textbf{39} (3), 779–856 (2011). 

\bibitem[Ha84]{ha84} E.\ Haeusler. A note on the rate of convergence in the martingale central limit theorem. \emph{Ann.\ Probab.} \textbf{12} (2), 635-639 (1984).

\bibitem[Ha88]{ha88} E.\ Haeusler. On the rate of convergence in the central limit theorem for martingales with discrete and continuous time. \emph{Ann.\ Probab.} \textbf{16} (1), 275-299 (1988).

\bibitem[HJ88]{hj} E.\ Haeusler, K.\ Joos. A nonuniform bound on the rate of convergence in the martingale central limit theorem \emph{Ann.\ Probab.} \textbf{16} (4), 1699-1720 (1988).

\bibitem[HB70]{hb70} C.C.~Heyde, B.M.~Brown. On the departure from normality of a certain class of martingales. \emph{Ann. Math. Statist.} \textbf{41}, 2161-2165 (1970).

\bibitem[JS]{js} J.\ Jacod, A.N.\ Shiryaev. \emph{Limit theorems for stochastic processes} (2nd edition). Grundlehren der mathematischen Wissenschaften \textbf{288}. Springer (2003).

\bibitem[JKO]{jko} V.V.~Jikov, S.M.~Kozlov, O.A.~Oleinik. \emph{Homogenization of differential operators and integral functionals}. Translated from Russian by G.A.\ Yosifian. Springer (1994).

\bibitem[Ka]{kato} T.\ Kato. \emph{Perturbation theory for linear operators} (2nd edition). Grundlehren der mathematischen Wissenschaften \textbf{132}, Springer (1976).

\bibitem[KV86]{kipvar} C.\ Kipnis, S.R.S.\ Varadhan. Central limit theorem for additive functionals of reversible Markov processes and applications to simple exclusions, \emph{Comm. Math. Phys.} \textbf{104}, 1-19 (1986).

\bibitem[Ko78]{kozlov1} S.M.~Kozlov. Averaging of random structures. \emph{Dokl. Akad. Nauk SSSR} \textbf{241} (5), 1016-1019 (1978). English transl.~: \emph{Soviet Math. Dokl.} \textbf{19} (4), 950-954  (1978).

\bibitem[Kü83]{kun} R.\ Künnemann. The diffusion limit for reversible jump processes on $\Z^d$ with ergodic random bond conductivities. \emph{Comm. Math. Phys.} \textbf{90} (1), 27–68 (1983). 

\bibitem[Le01]{lejay} A.\ Lejay. Homogenization of divergence-form operators with lower-order terms in random media. \emph{Probab. Theory Related Fields} \textbf{120} (2), 255–276 (2001). 

\bibitem[Ma]{maxwell} J.C.~Maxwell. Medium in which small spheres are uniformly disseminated. \emph{A treatise on electricity and magnetism}, 3d ed., part II, chapter IX, article 314. Clarendon press (1891).

\bibitem[Mo11]{berry} J.-C.\ Mourrat. A quantitative central limit theorem for the random walk among random conductances. Preprint, arXiv:1105.4485v1.

\bibitem[Mo12]{martingale_CLT} J.-C.\ Mourrat. On the rate of convergence in the martingale central limit theorem. \emph{Bernoulli}, to appear.

\bibitem[PV81]{papavara1} G.C.~Papanicolaou, S.R.S.~Varadhan. Boundary value problems with rapidly oscillating random coefficients. \emph{Random fields (Esztergom, 1979)} 835-873, Colloq. Math. Soc. J\'anos Bolyai \textbf{27}, North-Holland (1981). 

\bibitem[RR]{raru} S.T.\ Rachev, L.\ Rüschendorf. \emph{Mass transportation problems -- Vol.\ II: applications}. Probability and its applications, Springer (1998).

\bibitem[RS]{rs-fourier} M.\ Reed, B.\ Simon. \emph{Methods of modern mathematical physics II -- Fourier analysis, self-adjointness}. Academic press (1975).

\bibitem[SS04]{sid} V.\ Sidoravicius, A.-S.\ Sznitman. Quenched invariance principles for walks on clusters of percolation or among random conductances. \emph{Probab. Theory Related Fields} \textbf{129} (2), 219-244 (2004).

\bibitem[St]{rayleigh} J.W.~Strutt (3d Baron Rayleigh). On the influence of obstacles arranged in rectangular order upon the properties of a medium. \emph{Philos. mag.} \textbf{34}, 481-502 (1892).

\bibitem[Vi]{villani} C.\ Villani. \emph{Topics in optimal transportation.}
Graduate studies in mathematics \textbf{58}, American mathematical society (2003).

\bibitem[Yu80]{yuri1} V.V.~Yurinski\u{\i}. On a Dirichlet problem with random coefficients. \emph{Stochastic differential systems (Proc. IFIP-WG 7/1 Working Conf., Vilnius, 1978)} 344-353, Lecture Notes in Control and Information Sci. \textbf{25}, Springer (1980).

\bibitem[Yu86]{yuri} V.V.~Yurinski\u{\i}. Averaging of symmetric diffusion in a random medium (in Russian). \emph{Sibirsk. Mat. Zh.} \textbf{27} (4), 167-180 (1986). English transl. in \emph{Siberian Math. J.} \textbf{27} (4), 603-613 (1986).

\end{thebibliography}
\end{document}